\documentclass[11pt]{amsart}
\usepackage{amsfonts,amssymb,amscd,amsmath,verbatim,calc,graphicx}

\usepackage[margin=2.4cm]{geometry}
\usepackage[usenames,dvipsnames]{color}
\usepackage[shortlabels]{enumitem}
\usepackage{tikz, booktabs, colortbl, here}

\def\NZQ{\mathbb}               
  \def\ZZ{{\NZQ Z}}   
\def\opn#1#2{\def#1{\operatorname{#2}}} 
\opn\GL{GL} \opn\auto{Aut} \opn\inn{Inn} \opn\fix{Fix} \opn\ord{ord} \opn\id{id}
   \def\Qc{{\mathcal Q}}
\newcommand{\Sf}{\mathfrak{S}}
\newcommand{\Af}{\mathfrak{A}}
\newtheorem{thm}{Theorem}[section]
\newtheorem{lem}[thm]{Lemma}
\newtheorem{prop}[thm]{Proposition}

\newtheorem{q}[thm]{Question}
\newtheorem{prob}[thm]{Problem}
\theoremstyle{definition}
\newtheorem{defi}[thm]{Definition}

\newtheorem*{acknowledgement}{Acknowledgments}
\theoremstyle{remark}
\newtheorem{rem}[thm]{Remark}
\opn\proj{Proj}
\def\Ic{{\mathcal I}}
\def\Oc{{\mathcal O}}

\begin{document}

\title{Homogeneous quandles arising from automorphisms of symmetric groups}

\author{Akihiro Higashitani}
\author{Hirotake Kurihara}
\address{Akihiro Higashitani, Department of Pure and Applied Mathematics, Graduate School of Information Science and Technology, Osaka University, 
Yamadaoka 1-5, Suita, Osaka 565-0871, Japan}
\email{higashitani@ist.osaka-u.ac.jp}
\address{Hirotake Kurihara, Department of Integrated Arts and Science Kitakyushu National College of Technology, 
5-20-1 Shii, Kokuraminami-ku, Kitakyushu, Fukuoka 802-0985, Japan}
\email{kurihara@kct.ac.jp}

\begin{abstract}
Quandle is an algebraic system with one binary operation, but it is quite different from a group. 
Quandle has its origin in the knot theory and good relationships with the theory of symmetric spaces, 
so it is well-studied from points of view of both areas. 
In the present paper, we investigate a special kind of quandles, called generalized Alexander quandles $Q(G,\psi)$, 
which is defined by a group $G$ together with its group automorphism $\psi$. 
We develop the quandle invariants for generalized Alexander quandles. 
As a result, we prove that there is a one-to-one correspondence between generalized Alexander quandles arising from symmetric groups $\Sf_n$ 
and the conjugacy classes of $\Sf_n$ for $3 \leq n \leq 30$ with $n \neq 6,15$, and the case $n=6$ is also discussed. 
\end{abstract}

\maketitle

\section{Introduction}

Originally, quandles were introduced by Joyce \cite{Joyce} in the context of the knot theory. 
We call a set $Q$ a {\em quandle} if $Q$ is equipped with a binary operation $*$ satisfying the following three axioms: 
\begin{itemize}
\item[(Q1)] $x * x = x$ for any $x \in Q$; 
\item[(Q2)] for any $x,y \in Q$, there exists a unique $z \in Q$ such that $z * y = x$; 
\item[(Q3)] for any $x,y,z \in Q$, we have $(x*y)*z=(x*z)*(y*z)$. 
\end{itemize}
These axioms are derived from the Reidemeister moves appearing in the knot theory. 
We can rephrase this definition of quandles in terms of ``globally defined maps'' as follows. 
We call a set $Q$ a quandle if $Q$ is equipped with a ``point symmetry'' $s_x : Q \rightarrow Q$ defined in each $x \in Q$ satisfying the following three axioms: 
\begin{itemize}
\item[(Q1')] $s_x(x)=x$ for any $x \in Q$; 
\item[(Q2')] $s_x$ is a bijection on $Q$ for any $x \in Q$; 
\item[(Q3')] $s_x \circ s_y = s_{s_x(y)} \circ s_x$ holds for any $x,y\in Q$. 
\end{itemize}
It is straightforward to check that those definitions are equivalent by setting $s_x(\cdot) = (\cdot) * x$ for each $x \in Q$. 
We can see from this definition that every symmetric space has a quandle structure (cf.~\cite{Joyce}), 
where a {\em symmetric space} is a (Riemannian) manifold $M$ equipped with a globally defined map $s_x : M \rightarrow M$ for any $x \in M$ 
satisfying certain conditions, which are similar to (Q1')--(Q3'). 
Hence, quandles can be regarded as a ``discrete version'' of symmetric spaces. 
Quandles have been studied from this point of view, i.e., there are many results on quandles arising from the theory of symmetric spaces. 
Among those studies, a notion of \textit{homogeneous quandles} appears (see Section~\ref{sec:homogeneous}) and those are well-studied. (See, e.g., \cite{IT, KTW, Ven}, and so on.)

Although quandles have good relationships with the knot theory and the theory of symmetric spaces, quandles might be still mysterious when we treat them as a set equipped with one binary operation. 
Hence, it is quite natural to compare quandles with a very well-studied algebraic system, groups. 
On the other hand, we can construct a homogeneous quandle from a group together with its group automorphism. Such quandle is called a \textit{generalized Alexander quandle} (see Subsection~\ref{subsec:genAlex}). 
Our motivation to organize this paper is to find some connections with groups and homogeneous quandles arising from groups. 
More precisely, what we would like to do is to give a certain solution for Problem~\ref{toi}. 
In the present paper, for this purpose, we develop several quandle invariants of generalized Alexander quandles (see Section~\ref{sec:inv}). 
Moreover, as the first step, we investigate Problem~\ref{toi} in the case of symmetric groups. 
As a consequence, we obtain that generalized Alexander quandles arising from symmetric groups $\Sf_n$ one-to-one correspond to conjugacy classes of automorphism groups of $\Sf_n$ 
if $n \in \{3,4,\ldots,30\} \setminus \{15\}$ (Theorem~\ref{thm:sym}). 
Finally, we suggest the problem whether this is always true for any $n$ (Question~\ref{yosou}). 

The present paper is organized as follows. In Section~\ref{sec:notation}, we fix our notation for both quandles and groups. 
In Section~\ref{sec:homogeneous}, we recall from \cite{IT} what homogeneous quandles are. Homogeneous quandles can be obtained from, so-called, \textit{quandle triplets} (see Subsection~\ref{subsec:quandle}). 
As a special kind of homogeneous quandle, we introduce generalized Alexander quandles (see Subsection~\ref{subsec:genAlex}) and propose Problem~\ref{toi}. 
In Section~\ref{sec:inv}, for the investigation of Problem~\ref{toi}, we develop some invariants on generalized Alexander quandles. 
In Section~\ref{sec:sym}, we study Problem~\ref{toi} in the case of symmetric groups.

\begin{acknowledgement}
The authors would like to thank Hiroshi Tamaru and Takayuki Okuda for many helpful comments and stimulating discussions. 
The authors are partially supported by JSPS Grant-in-Aid for Young Scientists (B) $\sharp$16K17604. 
\end{acknowledgement}

\bigskip


\section{Terminologies and notation}\label{sec:notation}

Throughout this paper, we denote by $(Q,s)$ (or $Q$, for short) a quandle $Q$ equipped with the maps $\{s_x : x \in Q\}$ satisfying (Q1')--(Q3').

\subsection{Terminologies for quandles}

Let $(Q,s)$ and $(Q',s')$ be quandles. 
A map $f:Q \rightarrow Q'$ is called a {\em quandle homomorphism} if $f$ satisfies 
$$f \circ s_x = s_{f(x)}' \circ f \;\; \text{ for any } x \in Q.$$
Moreover, when $f$ is bijective, $f$ is said to be a {\em quandle automorphism}. 
If there is a quandle automorphism between $Q$ and $Q'$, then we say that $Q$ and $Q'$ are {\em isomorphic as quandles}, 
denoted by $Q \cong_Q Q'$. 

Let $\auto(Q,s)$ (or $\auto(Q)$, for short) be the set of quandle automorphisms from $(Q,s)$ to $(Q,s)$ itself, 
which is called the {\em automorphism group} of $Q$. 
Remark that $s_x \in \auto(Q)$ for any $x \in Q$. 
Let $\inn(Q,s)$ (or $\inn(Q)$, for short) be the subgroup of $\auto(Q)$ generated by $\{s_x : x \in Q\}$, which is called the {\em inner automorphism group} of $Q$. 

\subsection{Terminologies for groups}\label{sec:Term_group}

Let $G$ be a group with its unit $e$. Let $\auto(G)$ denote the automorphism group of $G$. 
For two groups $G$ and $G'$, we denote by $G \cong_G G'$ if $G$ and $G'$ are isomorphic as groups. 

For $g,h \in G$, we use a usual notation $g^h=hgh^{-1}$. 
Also, we use $g^G$, which stands for the conjugacy class with respect to $g \in G$. 
Let $\inn(G)$ denote the inner automorphism group of $G$, i.e., the set of the group automorphisms 
defined by $x \mapsto x^g$ for some $g \in G$, denoted by $(\cdot)^g$. 
It is well known that $\inn(G) \cong_G G/Z(G)$, where $Z(G)$ denotes the center of $G$ (e.g. \cite[Theorem 7.1]{Rotman}). 
Note that $\inn(G)$ is a normal subgroup of $\auto(G)$. 
We call $\psi$ an \textit{outer automorphism} if $\psi \in \auto(G) \setminus \inn(G)$.

Given $\psi \in \auto(G)$, let $$\fix(\psi,G)=\{g \in G : \psi(g)=g\}.$$ 
Note that $\fix(\psi,G)$ is a subgroup of $G$ and when we consider $(\cdot)^g \in \inn(G)$, we have $\fix((\cdot)^g,G)=C_G(g)$, 
where $C_G(g)=\{x \in G : gx = xg\}$ is the centralizer of $g$. 

\bigskip


\section{Homogeneous quandles and Quandle triplets}\label{sec:homogeneous}

In this section, we recall the notion of homogeneous quandles and study the relationship between homogeneous quandles and quandle triplets. 
These notions have their origin in \cite[Section 7]{Joyce} and the discussions were developed in \cite[Section 3]{IT} in detail.

\subsection{Correspondence between homogeneous quandles and quandle triplets}\label{subsec:quandle}

\begin{defi}[{\cite{IT}}]
Let $Q$ be a quandle. 
We say that $Q$ is {\em homogeneous} if $\auto(Q)$ acts transitively on $Q$, i.e., for any $x,y \in Q$, 
there exists $f \in \auto(Q)$ such that $f(x) = y$. 
\end{defi}

\begin{defi}[{\cite[Definition 3.1]{IT}}] Let $G$ be a group, let $K$ be a subgroup of $G$ and let $\psi \in \auto(G)$. 
We say that $(G,K,\psi)$ is a {\em quandle triplet} if $K \subset \fix(\psi,G)$. \end{defi}

In \cite[Section 3]{IT}, a correspondence between homogeneous quandles and quandle triplets is provided. 
Let us recall how to construct the homogeneous quandle from a quandle triplet. 
Let $(G,K,\psi)$ be a quandle triplet. We define a quandle as follows. 
We denote by $G/K=\{[g] : g \in G\}$ the set of left cosets with respect to $K$, and we set $$s_{[g]}([h]):=[g\psi(g^{-1}h)]$$ 
for $[g],[h] \in G/K$. It is proved in \cite[Proposition 3.2]{IT} that the set $G/K$ 
equipped with this map becomes a homogeneous quandle. 
Let us denote the homogeneous quandle constructed in this way by $Q(G,K,\psi)$.

\subsection{Generalized Alexander quandles}\label{subsec:genAlex}

Let $G$ be a group with its unit $e$. Given a group automorphism $\psi \in \auto(G)$, 
a triplet $(G,\{e\},\psi)$ trivially becomes a quandle triplet. 
Let $Q(G,\psi)=Q(G,\{e\},\psi)$, i.e., $Q(G,\psi)$ is the homogeneous quandle $(G,s)$, where $s$ is defined by 
$$s_g(h):=g\psi(g^{-1}h)\;\;\text{ for any }g,h \in G.$$ 
This quandle $Q(G,\psi)$ is also known as a {\em generalized Alexander quandle}. 
Note that in the case $G$ is a cyclic group, $Q(G,\psi)$ is called an {\em Alexander quandle} (see \cite{Nelson}).

One reason to study the generalized Alexander quandles (i.e., to specialize the case $K=\{e\}$) can be seen by the following. 
\begin{prop}\label{prop:divided_by_fix} Let $\psi,\psi' \in \auto(G)$ and let $K = \fix(\psi,G)$ and $K'=\fix(\psi',G)$. 
If $Q(G,\psi) \cong_Q Q(G,\psi')$, then $Q(G,K,\psi) \cong_Q Q(G,K',\psi')$. \end{prop}
\begin{proof}
Let $f:Q(G,\psi) \rightarrow Q(G,\psi')$ be a quandle automorphism. We define a map 
$\widetilde{f}:Q(G,K,\psi) \rightarrow Q(G,K',\psi')$ by $\widetilde{f}([x]_K):=[f(x)]_{K'}$ for $[x]_K \in G/K$, 
where $[ \cdot ]_K$ and $[ \cdot ]_{K'}$ stand for the left cosets with respect to $K$ and $K'$, respectively. 
In what follows, we will prove that $\widetilde{f}$ is a quandle automorphism. 

\noindent
{\bf Well-definedness}: Assume $[x]_K=[y]_K$ for $[x]_K,[y]_K \in G/K$. 
Then $x^{-1}y \in K$. Thus, $x^{-1}y=\psi(x^{-1}y)$. Hence, $y=x \cdot x^{-1}y = x \psi(x^{-1}y)=s_x(y)$. 
Moreover, since $f \circ s_x(y)=s_{f(x)}' \circ f(y)$ and $s_x(y)=y$, we obtain that $f(y)=f(x)\psi'(f(x)^{-1}f(y))$. 
Thus, $f(x)^{-1}f(y)=\psi'(f(x)^{-1}f(y))$. Hence, $f(x)^{-1}f(y) \in K'$, i.e., $[f(x)]_{K'}=[f(y)]_{K'}$. 
Therefore, $\widetilde{f}([x]_K)=\widetilde{f}([y]_K)$.

\noindent
{\bf Quandle homomorphism}: We see that \begin{align*}
\widetilde{f} \circ s_{[x]_K}([y]_K)&=\widetilde{f}([x\psi(x^{-1}y)]_K)=[f \circ s_x(y)]_{K'}=[s_{f(x)}' \circ f(y)]_{K'} \\
&=[f(x)\psi'(f(x)^{-1}f(y))]_{K'} =s_{\widetilde{f}([x]_K)}' \circ \widetilde{f}([y]_K). \end{align*}

\noindent
{\bf Bijectivity}: Let $\widetilde{f}^{-1}:Q(G,K',\psi') \rightarrow Q(G,K,\psi)$ by $\widetilde{f}^{-1}([y]_{K'}):=[f^{-1}(y)]_K$ for $[y]_{K'} \in G/K'$. 
Then $\widetilde{f}^{-1}$ is also well-defined and this implies the bijectivity of $\widetilde{f}$. 
\end{proof}

Given a finite group $G$, let $\Qc(G)$ be the set of quandle isomorphic classes of $Q(G,\psi)$'s, i.e., 
$$\Qc(G):=\{Q(G,\psi) : \psi \in \auto(G)\}/\cong_Q.$$

The following problem naturally arises. 
\begin{prob}\label{toi}
Determine $\Qc(G)$ for a given group $G$. 
\end{prob}

The following proposition says that we can roughly classify $Q(G,\psi)$ for $\psi \in \auto(G)$ up to quandle isomorphism 
by seeing the conjugacy classes of $\auto(G)$. 
\begin{prop}\label{prop:conj}
Let $\psi,\psi' \in \auto(G)$ and assume that $\psi$ and $\psi'$ are conjugate. Then we have $Q(G,\psi) \cong_Q Q(G,\psi')$. 	
\end{prop}
\begin{proof}
Let $(Q,s)=Q(G,\psi)$ and $(Q',s')=Q(G,\psi')$. 
Let $\psi'=\psi^\tau$ for some $\tau \in \auto(G)$. Clearly, $\tau$ gives a bijection between $Q$ and $Q'$ (note that $Q=Q'=G$ as sets). 
It is enough to show that $\tau$ is a quandle homomorphism, i.e., $\tau \circ s_x = s_{\tau(x)}' \circ \tau.$

Given any $y \in G$, we see that 
\begin{align*}
s_{\tau(x)}' \circ \tau(y)&=\tau(x)\psi'(\tau(x)^{-1}\tau(y)) =\tau(x)\psi'(\tau(x^{-1}y)) =\tau(x)\tau\circ\psi\circ\tau^{-1}(\tau(x^{-1}y)) \\
&=\tau(x)\tau(\psi(x^{-1}y)) =\tau(x\psi(x^{-1}y))=\tau \circ s_x(y), \end{align*}
as required. 
\end{proof}

Let $C_n$ be the cyclic group of order $n$. 
\begin{rem}[{\cite{Nelson}}]\label{rem:cyclic}
In the case where $G=C_n$, Problem~\ref{toi} is completely solved in \cite{Nelson} as follows. 
It is well known that $\auto(C_n) \cong_G U(C_n)$ (see \cite[Lemma 7.2]{Rotman}), where $U(C_n)=\{x \in C_n : x \text{ is coprime to }n\}$ is the group of units of $C_n$. 
More precisely, given $a \in U(C_n)$, the automorphism of $C_n$ is defined by $x \mapsto ax$ for each $x \in C_n$. 
It is proved in \cite[Corollary 2.2]{Nelson} that $Q(C_n, a) \cong_Q Q(C_n,b)$ for $a,b \in U(C_n)$ if and only if 
$N(n,a)=N(n,b)$ and $a \equiv b \pmod{N(n,a)}$, where $N(n,a)=\frac{n}{\mathrm{gcd}(n,1-a)}$. 
Namely, $\Qc(C_n)$ is completely characterized. 

For example, $Q(C_9,4) \cong_Q Q(C_9,7)$ since $N(9,4)=N(9,7)=3$ and $4 \equiv 7 \pmod{3}$. 
On the other hand, since $U(C_n)$ is abelian, the set of the conjugacy classes of $U(C_n)$ is nothing but $U(C_n)$ itself. 
Hence, this example shows that the set of the conjugacy classes of $\auto(C_n)$ does not necessarily one-to-one correspond to $\Qc(C_n)$. 
\end{rem}

\bigskip


\section{Invariants on homogeneous quandles}\label{sec:inv}

Towards a solution of Problem \ref{toi}, we establish some quandle invariants for $Q(G,\psi)$. Throughout this section, we fix a finite group $G$.

At first, we check that $\inn(Q,s)$ is a quandle invariant as follows. 
\begin{prop}\label{meidai1}
A quandle automorphism $f:(Q,s) \rightarrow (Q',s')$ induces a group automorphism of two groups $\inn(Q,s)$ and $\inn(Q',s')$. 
Namely, the map defined by 
$$f:\inn(Q,s) \rightarrow \inn(Q',s'), \; s_g \mapsto s_{f(g)}'$$
which is extended as a group homomorphism is a group automorphism. 
In particular, if $(Q,s) \cong_Q (Q',s')$, then $\inn(Q,s) \cong_G \inn(Q',s')$. 
\end{prop}
\begin{proof}
This $f$ is obviously a group homomorphism and bijective if it is well-defined. Thus, what we have to show is the well-definedness. 

Let $s_{g_1} \circ \cdots s_{g_k} = s_{h_1} \circ \cdots \circ s_{h_\ell} \in \inn(Q)$. 
Then $s_{g_1} \circ \cdots s_{g_k} \circ f^{-1}= s_{h_1} \circ \cdots \circ s_{h_\ell} \circ f^{-1}$ as 
quandle homomorphisms from $(Q',s')$ to $(Q,s)$. Moreover, since $f$ is a quandle homomorphism, we see that \begin{align*}
s_{g_1} \circ \cdots s_{g_k} \circ f^{-1} &= f^{-1} \circ s_{f(g_1)}' \circ \cdots \circ s_{f(g_k)}' \text{ and }\\
s_{h_1} \circ \cdots s_{h_\ell} \circ f^{-1} &= f^{-1} \circ s_{f(h_1)}' \circ \cdots \circ s_{f(h_\ell)}'. \end{align*}
Therefore, \begin{align*}
\underbrace{f \circ f^{-1}}_{\id} \circ s_{f(g_1)}' \circ \cdots \circ s_{f(g_k)}' = 
\underbrace{f \circ f^{-1}}_{\id} \circ s_{f(h_1)}' \circ \cdots \circ s_{f(h_\ell)}'. \end{align*}
\end{proof}

Next, we focus on the order of each element. Let $\ord_G(g)$ denote the order of $g \in G$. 
Similarly, for a quandle $(Q,s)$, let $\ord(s_x)=\min\{n \in \ZZ_{>0} : \underbrace{s_x \circ \cdots \circ s_x}_n=\id\}$. 
For a homogeneous quandle $Q$, we see that $\ord(s_x)$ is constant for any $x \in Q$. 
In fact, for any $x,y \in Q$, since there is $f \in \auto(Q)$ such that $y=f(x)$, we see that 
if $s_x^m=\id$, then $f=f \circ s_x^m=s_{f(x)}^m \circ f$, so we obtain $s_{f(x)}^m=\id$, as required. 
Thus, we use the notation $\ord(Q)$ instead of $\ord(s_x)$ for $x \in Q$ in the case $Q$ is homogeneous. 
Moreover, we also see that $\ord(Q)$ is a quandle invariant. In the case of $Q(G,\psi)$, we can compute this as follows. 

\begin{prop}
One has $\ord(Q(G,\psi))=\ord_{\auto(G)}(\psi)$. 
\end{prop}
\begin{proof}
Our goal is to show that $s_x^m=\id$ implies $\psi^m=\id$, and vice versa. 
A direct computation shows that 
\begin{equation}\label{eq:s_x^i}
\begin{split}
s_x^m(y)&=s_x^{m-1}(x\psi(x^{-1}y))=s_x^{m-2}(x\psi(x^{-1}x\psi(x^{-1}y)))=s_x^{m-2}(x\psi^2(x^{-1}y))=\cdots\\
&=x\psi^m(x^{-1}y)
\end{split}
\end{equation}
for any $y \in G$. 
\begin{itemize}
\item If $\psi^m=\id$, then $s_x^m(y) = y$. 
\item If $s_x^m=\id$, since $x\psi^m(x^{-1}y)=y$ for any $y \in G$, we have $\psi^m(x^{-1}y)=x^{-1}y$ for any $y \in G$. 
This implies that $\psi^m=\id$. 
\end{itemize}
\end{proof}

Let $K$ be a subgroup of $G$ and let $[G:K]$ denote the index of $K$. 
\begin{prop}There is a one-to-one correspondence 
between $\{s_x : x \in Q(G,\psi)\}$ and $G/\fix(\psi,G)$. 
In particular, we have $|\{s_x : x \in Q(G,\psi)\}|=[G:\fix(\psi,G)]$. \end{prop}
\begin{proof}
We will show that $\{ s_x : x \in G\}$ one-to-one corresponds to the left cosets $\{ [x] : x \in G\}$ with respect to $\fix(\psi,G)$. 
The bijectivity of those sets can be verified as follows: 
\begin{align*}
s_x(y)=s_{x'}(y) \;\text{ for any }y \in G &\Longleftrightarrow \; x\psi(x^{-1}y)=x'\psi(x'^{-1}y) \;\text{ for any }y \in G \\
&\Longleftrightarrow x^{-1}x'=\psi(x^{-1}y)\psi(x'^{-1}y)^{-1}\;\text{ for any }y \in G \\
&\Longleftrightarrow x^{-1}x'=\psi(x^{-1}x')\; \Longleftrightarrow \; x^{-1}x' \in \fix(\psi,G) \\
&\Longleftrightarrow [x]=[x']. 
\end{align*}
\end{proof}

For $i \in \ZZ_{>0}$ and for a quandle $(Q,s)$, 
let $(Q^{(i)},s^{(i)})$ (or $Q^{(i)}$, for short) be defined by $Q^{(i)}=Q$ as a set and $s_x^{(i)}:=\underbrace{s_x \circ \cdots \circ s_x}_i$. 
\begin{prop}\label{prop:i_bai}
    Work with the notation as above.
    \begin{enumerate}[label=$($\alph*$)$]
        \item Let $Q$ and $R$ be quandles. Assume that $Q \cong_Q R$. Then $Q^{(i)} \cong_Q R^{(i)}$ for any $i$.
        \item Fix $\psi \in \auto(G)$ and let $Q=Q(G,\psi)$. Then $Q^{(i)}=Q(G,\psi^i)$. 
    \end{enumerate}
\end{prop}
\begin{proof}
(a) Let $Q=(Q,s)$ and $R=(R,s')$ and let $f : Q \rightarrow R$ be a quandle automorphism. 
Then the same $f$ gives a bijection between $Q^{(i)}$ and $R^{(i)}$. We may show that $f$ is a quandle homomorphism between them, but it is straightforward: 
\begin{align*}
f \circ s_x^{(i)}= f \circ \underbrace{s_x \circ \cdots \circ s_x}_i = \underbrace{s_{f(x)}' \circ \cdots \circ s_{f(x)}'}_i \circ f = s_{f(x)}'^{(i)} \circ f. 
\end{align*}

(b) The equality $Q(G,\psi)^{(i)}=Q(G,\psi^i)$ directly follows from \eqref{eq:s_x^i}. 
\end{proof}

We summarize Proposition~\ref{prop:conj} and the propositions proved in this section. 
\begin{thm}\label{thm:main}
Let $\psi,\psi' \in \auto(G)$ and let $Q=Q(G,\psi)$ and let $Q'=Q(G,\psi')$. 
\begin{enumerate}[label=$($\alph*$)$]
    \item If $\psi'=\psi^\tau$ for some $\tau \in \auto(G)$,
     then $Q \cong_Q Q'$. 
     \item If $Q \cong_Q Q'$, then 
     \begin{itemize}
        \item $\ord_{\auto(G)}\psi=\ord_{\auto(G)}\psi'$; 
        \item $[G:\fix(\psi,G)]=[G:\fix(\psi',G)]$, i.e., $|\fix(\psi,G)|=|\fix(\psi',G)|$; 
        \item $\inn(Q) \cong_G \inn(Q')$; 
        \item $Q(G,\psi^i) \cong_Q Q(G,\psi'^i)$ for any $i \in \ZZ_{>0}$. 
        \end{itemize}        
\end{enumerate}
\end{thm}

\bigskip


\section{The case of symmetric groups}\label{sec:sym}

Let $\Sf_n$ (resp. $\Af_n$) denote the symmetric (resp. alternating) group on $\{1,\ldots,n\}$. 
In this section, we develop the invariants appearing in Theorem~\ref{thm:main} for generalized Alexander quandles arising from $\Sf_n$. 
After preparing the invariants, we discuss Problem~\ref{toi} for $\Sf_n$. 

For the fundamental materials on $\Sf_n$, see, e.g., \cite{Rotman}. 

\subsection{Fundamental facts on $\Sf_n$}
We first recall the structure of $\auto(\Sf_n)$. 
\begin{prop}[{cf. \cite[Section 7]{Rotman}}]\label{prop:auto_S_n}
We have $$\auto(\Sf_n)=\inn(\Sf_n) $$ 
if $n \neq 2,6$, $\auto(\Sf_2)=\{\mathrm{id}\}$, and $\auto(\Sf_6) \cong_G \Sf_6 \rtimes C_2$. 
\end{prop}

In Subsection~\ref{sec:S_6}, we will see the structure of $\auto(\Sf_6)$. 
As we mentioned in Subsection~\ref{sec:Term_group}, we have $\inn(G) \cong_G G/Z(G)$. 
Since $Z(\Sf_n)=\{e\}$ (see, e.g., \cite[Exercise 3.1 (i)]{Rotman}), we obtain that $\auto(\Sf_n) = \inn(\Sf_n) \cong_G \Sf_n$ when $n \neq 2,6$.

Let us recall a system of generators of $\Sf_n$ and $\Af_n$, respectively. The following facts are well known. 
\begin{lem}[{cf.~\cite[Theorem~6.6 and Exercise~7 in Chapter~6]{Roman}}]
    \label{lem:gene_S_A}
    The following assertions hold: 
    \begin{enumerate}[label=$($\alph*$)$]
        \item 
        \label{enu:gene_S}
        $\{(1\; 2),(1\; 3),\ldots ,(1\; n)\}$ and $\{(1\; 2),(1\; 2\; 3\; \cdots \; n)\}$
        are systems of generators of $\Sf_n$, respectively. 
        \item
        \label{enu:gene_A}
        $\{(1\; 2\; 3),(1\; 2\; 4),\ldots ,(1\; 2\; n)\}$ is a system of generators of $\Af_n$. 
    \end{enumerate}
\end{lem}

Moreover, we also see the following which we will use in the proof of Proposition~\ref{prop:innQ}. 
\begin{lem}\label{lem:center}
Assume $n \geq 5$. Then we have $$C_{\Sf_n}(\Af_n)=\{ x \in \Sf_n : xg=gx \text{ for any }g \in \Af_n\}=\{e\}.$$ 
\end{lem}
\begin{proof}
Let $n$ be odd. Let $\pi_1=(1 \ 2 \ 3 \ \cdots\ n)$ and $\pi_2=(1 \ 3 \ 2 \ 4 \ \cdots \ n)$. Then $\pi_1,\pi_2 \in \Af_n$. 
Since $C_{\Sf_n}(\pi_j)=\{\pi_j^i : i=0,1,\ldots,n-1\}$ for each $j=1,2$ and $\{e\} \subset C_{\Sf_n}(\Af_n) \subset C_{\Sf_n}(\pi_1) \cap C_{\Sf_n}(\pi_2) =\{e\}$ by $n \geq 5$, we conclude the assertion. 

Even in the case $n$ is even, by taking $\pi_1=(1 \ 2 \ 3 \ \cdots\ n-1)$ and $\pi_2=(1 \ 3 \ 2 \ 4 \ \cdots \ n-1)$, we can apply the same discussion. 
\end{proof}

The conjugacy classes of $\Sf_n$ one-to-one correspond to the partitions of $n$. 
Let us recall the correspondence. Given $\pi \in \Sf_n$, we can decompose $\pi$ into disjoint cyclic permutations $\pi=\pi_1 \pi_2 \cdots \pi_k$, 
where the order of $\pi_i$ is non-increasing. Let $\lambda_i$ be the order of $\pi_i$. 
Then $\lambda_1 \geq \cdots \geq \lambda_k \geq 1$ with $n=\sum_{i=1}^k \lambda_i$. 
We call $(\lambda_1,\ldots,\lambda_k)$ the \textit{shape} (also known as \textit{cycle structure}) of $\pi$. 
It is well known that $\pi$ and $\pi'$ are conjugate if and only if those have the same shape (\cite[Theorem 3.5]{Rotman}). 
Namely, the conjugacy classes of $\Sf_n$ (i.e., $\auto(\Sf_n)$ with $n \neq 2,6$) one-to-one correspond to the partitions of $n$. 
Note that the partition $(1,1,\ldots,1)$ corresponds to a trivial conjugacy class $e \in \Sf_n$. 
We sometimes use the notation $a^i$ if $a$ appears $i$ times in the partition, e.g., $(2^3,1^4)$ stands for $(2,2,2,1,1,1,1)$. 

Given $\pi \in \Sf_n$ with the shape $\lambda=(\lambda_1,\ldots,\lambda_k)$, we see that
\[
\ord_{\Sf_n}(\pi)=\mathrm{lcm}(\lambda_1,\ldots,\lambda_k) \;\; \text{and}\;\; 
|C_{\Sf_n}(\pi)|=\prod_{i = 1}^ki^{a_i} \cdot a_i!, 
\]
where $a_i=|\{j : \lambda_i=j\}|$. See, e.g., \cite[Theorem 6.2, Theorem 6.12, Exercise 19 in Chapter 6]{Roman}.

We say that $\pi \in \Sf_n$ is {\em even} (resp. {\em odd}) if $\pi$ is an even (resp. odd) permutation. 
We prepare the following lemma which will be used in the proof of Lemma~\ref{lem:inn}. 
\begin{lem}\label{lem:conj}
Assume $n\ge 5$. Fix $\pi \in \Sf_n$ with $\pi \neq e$. 
Then, for any $g \in \Af_n$, there exist even number of elements $x_1,\ldots,x_{2r} \in \Sf_n$ such that $g=\pi^{x_1}\pi^{x_2}\cdots\pi^{x_{2r}}$. 
\end{lem}
\begin{proof}
Write $\pi=\pi_1 \cdots \pi_k$ for the product of disjoint cyclic permutations. 
Remark that we omit cyclic permutations with length $1$ in this notation. Since $\pi \neq e$, we may assume that $\lambda_k \geq 2$. 
Moreover, it suffices to show the claim in the case 
$$\pi=(1 \ 2 \ \cdots \ \lambda_1)(\lambda_1 + 1 \ \cdots \ \lambda_1+\lambda_2)\cdots \left(\sum_{i=1}^{k-1}\lambda_i+1 \ \cdots \ \sum_{i=1}^k\lambda_i\right)$$
since we can get any $\pi$ by taking conjugation from this particular form. 

Let $\lambda_1=\lambda$. 
In the case $\lambda \geq 3$, we observe that $(1 \; \lambda \; \lambda-1 \; \cdots \; 2)(1 \; 3 \; 2 \; 4 \; 5 \; \cdots \; \lambda)=(1 \; 2 \; 3)$. 
Namely, there are $x,x' \in \Sf_n$ such that $\pi_1^x \cdot \pi_1^{x'}=(1 \; 2 \; 3)$, where 
$\pi_1^x=(1 \; \lambda \; \lambda-1 \; \cdots \; 2)$ and $\pi_1^{x'}=(1 \; 3 \; 2 \; 4 \; 5 \; \cdots \; \lambda)$. 
On the other hand, there is $y_i \in \Sf_n$ such that $\pi_i^{y_i} = \pi_i^{-1}$. 
Hence, by taking $z$ and $z'$ properly, we obtain that $\pi^z \cdot \pi^{z'}=(1 \; 2 \; 3)$. 
Moreover, $(\pi^x \cdot \pi^{x'})^{(3 \; j)}=\pi^{(3 \ j)x} \cdot \pi^{(3 \ j)x'}=(1 \; 2 \; j)$ for any $3 \leq j \leq n$. 
Note that $\Af_n$ is generated by $(1 \; 2 \; j)$ for $3 \leq j \leq n$ by Lemma~\ref{lem:gene_S_A} \ref{enu:gene_A}. 
Therefore, the desired conclusion follows. 

In the case $\lambda \leq 2$, we have $\ord(\pi_i) = 2$ for any $i$. Thus, $\pi=(1 \; 2) \cdots (2k-1 \; 2k)$. Suppose that $2k < n$. 
Then there are $x,x' \in \Sf_n$ such that $\pi^x=(1 \; 2)(4 \; 5) \cdots (2k \; 2k+1)$ and $\pi^{x'}=(2 \; 3)(4 \; 5) \cdots (2k \; 2k+1)$. 
Thus, $\pi^x \cdot \pi^{x'}=(1 \; 2 \; 3)$. Hence, similar to the above, we obtain the desired conclusion. 

In the case $\ord(\pi_1) = 2$ and $2k = n$, since $n \geq 5$, we have $k \geq 3$. Now we observe that 
$(1 \; 5)(3 \; 4)(2 \; 6)\cdot(1 \; 4)(2 \; 5)(3 \; 6)\cdot(1 \; 5)(3 \; 6)(2 \; 4)\cdot(1 \; 6)(2 \; 5)(3 \; 4) =(1 \; 2 \; 3)$. 
By taking $x,x',x'',x''' \in \Sf_n$ such that 
\begin{align*}
&\pi^x=(1 \; 5)(3 \; 4)(2 \; 6)(7 \; 8) \cdots (2k-1 \; 2k), \quad \pi^{x'}=(1 \; 4)(2 \; 5)(3 \; 6)(7 \; 8) \cdots (2k-1 \; 2k), \\
&\pi^{x''}=(1 \; 5)(3 \; 6)(2 \; 4)(7 \; 8) \cdots (2k-1 \; 2k), \quad \pi^{x'''}=(1 \; 6)(2 \; 6)(3 \; 4)(7 \; 8) \cdots (2k-1 \; 2k), 
\end{align*}
we obtain that $\pi^x \cdot \pi^{x'} \cdot \pi^{x''} \cdot \pi^{x'''} =(1 \; 2 \; 3)$, as required. 
\end{proof}

\subsection{The structure of $\inn(Q(\pi))$}

For the remaining parts of the paper, we use the abbreviation $Q(\Sf_n,(\cdot)^\pi)=Q(\pi)$. 
For each cyclic permutation $\gamma \in \Sf_n$, we see that $\ord_{\Sf_n}(\gamma)$ is odd if and only if $\gamma$ is an even permutation. 
This implies that for each $\pi \in \Sf_n$, if $\ord_{\Sf_n}(\pi)$ is odd, then $\pi$ should be an even permutation. 
Namely, if $x \in \Sf_n$ is an odd permutation, then $\ord_{\Sf_n}(x)$ should be even. 

Fix $\pi \in \Sf_n$ and consider $Q(\pi)$. Let $m=\ord(Q(\pi))=\ord_{\Sf_n}(\pi)$. 
\begin{lem}\label{lem:inn} Work with the notation as above. Assume $n \geq 5$. Then we have
$$\inn(Q(\pi)) \cong_G \Af_n \rtimes_\varphi C_m,$$ 
which is a semidirect product of $\Af_n$ and $C_m$ with $\varphi : C_m \rightarrow \auto(\Af_n)$, $\overline{i} \mapsto (\cdot)^{\pi^i}$, i.e., 
$$(g,\overline{i}) \cdot (h,\overline{j})=(g\pi^ih\pi^{-i},\overline{i+j}).$$
\end{lem}
\begin{proof}
We define the map $\Phi : \Af_n \rtimes_\varphi C_m \rightarrow \inn(Q(\pi))$ by $$\Phi((g,\overline{\ell}))(\cdot)=g \pi^\ell (\cdot) \pi^{-\ell}.$$
Our work is to show that this map is a group isomorphism. 

\noindent
{\bf Well-definedness}: For the well-definedness of $\Phi$, 
it suffices to show that a map $g \pi^\ell ( \cdot ) \pi^{-\ell}$ can be written as a composition of certain $s_x$'s for any $(g, \overline{\ell}) \in \Af_n \rtimes_\varphi C_m$. 
By Lemma~\ref{lem:conj}, 
there exist $x_1,\ldots,x_{2r} \in \Sf_n$ such that $g=\pi^{x_1}\cdots\pi^{x_{2r}}$. 
Moreover, there is $y \in \Sf_n$ with $\pi^y=\pi^{-1}$. Hence, $s_e \circ s_y(\cdot)=e(\cdot)\pi^{-2}$. 
Note that there is $t \in \ZZ_{>0}$ with $2t+2r+\ell \equiv \ell \pmod{m}$. Therefore, we conclude that 
$$(s_e \circ s_y)^t \circ s_{x_1} \circ \cdots \circ s_{x_{2r}} \circ s_e^\ell (\cdot)= e^tg \pi^\ell ( \cdot ) \pi^{-2t-2r-\ell} = g \pi^\ell (\cdot ) \pi^{-\ell}.$$

\noindent
{\bf Homomorphism}: 
Let $(g,\overline{i}), (h,\overline{j}) \in \Af_n \rtimes_\varphi C_m$. Then we see the following: 
\begin{align*}
\Phi((g,\overline{i}) \cdot (h,\overline{j}))&=\Phi((g\pi^ih\pi^{-i},\overline{i+j}))=g\pi^ih\pi^j (\cdot) \pi^{-i-j}, \;\;\text{and}\\
\Phi((g,\overline{i})) \circ \Phi((h,\overline{j}))&=g\pi^i(h\pi^j (\cdot)\pi^{-j})\pi^{-i}=g\pi^ih\pi^j (\cdot) \pi^{-i-j}. 
\end{align*}

\noindent
{\bf Injectivity}: Assume that $\Phi((g,\overline{i}))=\Phi((h,\overline{j}))$. 
Then $g\pi^i x \pi^{-i} = h\pi^j x \pi^{-j}$ holds for any $x \in \Sf_n$. By substituting $x=\pi^i$, we obtain that 
$$g\pi^i = h\pi^{j+i-j} \;\Longleftrightarrow g=h.$$
Hence, $\pi^i x \pi^{-i} = \pi^j x \pi^{-j}$ holds for any $x \in \Sf_n$. Since $Z(\Sf_n)=\{e\}$, we conclude that 
$$x \pi^{j-i} = \pi^{j-i} x \text{ for any } x \in \Sf_n \;\Longleftrightarrow\; \pi^{j-i} = e \;\Longleftrightarrow\; \pi^i=\pi^j \;\Longleftrightarrow\; i \equiv j \pmod{m}.$$ 

\noindent
{\bf Surjectivity}: 
Take $f=s_{x_1} \circ \cdots s_{x_\ell} \in \inn(Q(\pi))$ arbitrarily. Then $f=\pi^{x_1}\cdots\pi^{x_\ell}(\cdot)\pi^{-\ell}$. 
Here, we notice that $\pi^{x_1} \cdots \pi^{x_\ell} \pi^{-\ell} \in \Af_n$. 
Hence, $$\Af_n \rtimes_\varphi C_m \ni (\pi^{x_1} \cdots \pi^{x_\ell} \pi^{-\ell},\overline{\ell}) \mapsto \pi^{x_1}\cdots\pi^{x_\ell}\pi^{-\ell} \pi^\ell (\cdot)\pi^{-\ell}=f,$$ 
as required. 
\end{proof}

From the structure of $\inn(Q(\pi))$, we conclude the following: 
\begin{prop}\label{prop:innQ}
Work with the same notation as Lemma~\ref{lem:inn}.
\begin{enumerate}[label=$($\alph*$)$]
    \item If $\pi$ is even, then we have $\inn(Q(\pi)) \cong_G \Af_n \rtimes_\varphi C_m \cong_G \Af_n \times C_m$. 
    \item If $\pi$ is odd, then we have $\inn(Q(\pi)) \cong_G \Af_n \rtimes_\varphi C_m \not\cong_G \Af_n \times C_m$. 
\end{enumerate}
In particular, for $\pi,\pi' \in \Sf_n$, if $Q(\pi) \cong_Q Q(\pi')$, then $\pi$ and $\pi'$ have the same parity. 
\end{prop}
\begin{proof}
(a) We define the map $\Psi : \Af_n \rtimes_\varphi C_m \rightarrow \Af_n \times C_m$ by 
$$\Psi((g,\overline{\ell})) =(g \pi^\ell, \overline{\ell}).$$ 
Then it is straightforward to check that this gives a group isomorphism: 
\begin{itemize}
\item Since $\pi$ is even, we see that $g \pi^\ell \in \Af_n$ for any $g \in \Af_n$ and $\overline{\ell} \in C_m$. Hence, the well-definedness follows. 
\item One has $\Psi((g,\overline{i})\cdot(h,\overline{j}))=\Psi((g\pi^ih \pi^{-i},\overline{i+j}))=(g\pi^ih\pi^j,\overline{i+j})=\Psi((g,\overline{i}))\Psi((h,\overline{j})).$ 
\item One has $\Psi((g,\overline{i}))=\Psi((h,\overline{j})) \;\Longleftrightarrow\;(g\pi^i,\overline{i})=(h\pi^j,\overline{j})  \;\Longleftrightarrow\; g=h \text{ and } \overline{i}=\overline{j}.$ 
\item For any $(g,\overline{i}) \in \Af_n \times C_m$, since $g \pi^{-i} \in \Af_n$ by our assumption, we have $\Psi((g\pi^{-i},\overline{i}))=(g,\overline{i})$, as required. 
\end{itemize}

(b) Suppose that there exists a group isomorphism
$\Omega \colon \Af_n \rtimes_\varphi C_m \to \Af_n \times C_m$. 
Let $X=\{(g, \overline{0}) : g \in \Af_n\} \subset \Af_n \rtimes_\varphi C_m$. 
Then $X$ is a subgroup of $\Af_n \rtimes_\varphi C_m$.
Let $X'=\Omega(X)$. Then $X'$ becomes a subgroup of $\Af_n \times C_m$. 
For any $(g, \overline{0}) \in X$,
let $(a_g,\overline{j_g})=\Omega(g,\overline{0})$.

Let $Y'=\{(e,\overline{i}) : \overline{i} \in C_m\}
\subset \Af_n \times C_m$.
Then $Y'$ is a subgroup of $\Af_n \times C_m$.
Let $Y=\Omega^{-1}(Y')$. 
Then $Y$ becomes a subgroup of $\Af_n \rtimes_\varphi C_m$ and $|Y|=|Y'|=m$ since $\Omega$ is a group isomorphism. 
For any $(e,\overline{i}) \in Y'$, let $(b_i,\overline{k_i})=\Omega^{-1}((e,\overline{i}))$. 

Since $\Omega^{-1}$ is also a group isomorphism, we see that 
\begin{align*}
\Omega^{-1} ((e,\overline{i}) (a_g,\overline{j_g}))&=\Omega^{-1} ((e,\overline{i})) \cdot \Omega^{-1} ((a_g,\overline{j_g}))=(b_i,\overline{k_i})\cdot(g,\overline{0})=(b_i\pi^{k_i} g \pi^{-{k_i}},\overline{k_i}) \text{ and }\\
\Omega^{-1} ((e,\overline{i}) (a_g,\overline{j_g}))&=\Omega^{-1} ((a_g,\overline{i+j_g}))=\Omega^{-1} ((a_g,\overline{j_g})(e,\overline{i}))
=\Omega^{-1} ((a_g,\overline{j_g}))\cdot \Omega^{-1} ((e,\overline{i}))\\
&=(g,\overline{0}) \cdot (b_i,\overline{k_i})=(gb_i,\overline{k_i}). 
\end{align*}
Hence, we obtain that $b_i\pi^{k_i} g \pi^{-k_i} = gb_i \;\Longleftrightarrow (b_i\pi^{k_i})g=g(b_i\pi^{k_i})$ for any $g \in \Af_n$. 
Here, we have $C_{\Sf_n}(\Af_n)=\{e\}$ by Lemma~\ref{lem:center}. Hence, $b_i \pi^{k_i} \in C_{\Sf_n}(\Af_n)=\{e\}$. Thus, $b_i = \pi^{-k_i}$. 
However, since $b_i \in \Af_n$ and $\pi$ is odd, we see that $k_i$ must be even for any $\overline{i} \in C_m$. This implies that 
$Y'=\{(b_i,\overline{k_i}) : \overline{i} \in C_m\}=\{(\pi^{-k},\overline{k}) : k \text{ is even}\}$, so $|Y'|=m/2$, a contradiction. 
(Note that $m$ is even since $\pi$ is an odd permutation.) 

Therefore,
$\Af_n \rtimes_\varphi C_m \not\cong_G \Af_n \times C_m$, as desired. 
\end{proof}

\subsection{The structure of $\auto(\Sf_6)$ and quandles arising from outer automorphisms}\label{sec:S_6}

Let $\mathcal{C}_\lambda$ denote the conjugacy class of $\Sf_n$ with respect to the shape $\lambda$. 

Firstly, we recall the structure of $\auto(\Sf_6)$. Originally, H\"{o}lder~\cite{Holder} proved that $\Sf_6$ has an outer automorphism
that is unique up to multiplication by an inner automorphism, and $\auto(\Sf_6)/\inn(\Sf_6)\cong_G C_2$. 
Actually, we saw in Proposition~\ref{prop:auto_S_n} that $\auto(\Sf_6)\cong_G \inn(\Sf_6) \rtimes C_2$. 
We describe the structure of this semidirect product. It follows that $\varphi \in \auto(\Sf_6)$ preserves $\mathcal{C}_{(2,1^4)}$
if and only if $\varphi \in \inn(\Sf_6)$ (see, cf.~\cite[Lemma~7.4]{Rotman}), 
and any outer automorphism of $\Sf_6$ swaps $\mathcal{C}_{(2,1^4)}$ and $\mathcal{C}_{(2^3)}$. 
In this paper, we define an outer automorphism $\xi:\Sf_6 \rightarrow \Sf_6$ as follows: 
\begin{align*}
    &(1 \; 2) \mapsto (1 \; 2)(3 \; 4)(5 \; 6), \\ 
    &(2 \; 3) \mapsto (1 \; 6)(2 \; 4)(3 \; 5), \\ 
    &(3 \; 4) \mapsto (1 \; 2)(3 \; 6)(4 \; 5), \\ 
    &(4 \; 5) \mapsto (1 \; 6)(2 \; 5)(3 \; 4), \\ 
    &(5 \; 6) \mapsto (1 \; 2)(3 \; 5)(4 \; 6), 
\end{align*}
and extend this into a group homomorphism. Remark that $\ord(\xi)=2$. In fact, 
\begin{align*}
    (1\; 2) 
    \overset{\xi}{\mapsto} (1 \; 2)(3 \; 4)(5 \; 6) 
    \overset{\xi}{\mapsto} (1 \; 2)(3 \; 4)(5 \; 6)\cdot 
    (1 \; 2)(3 \; 6)(4 \; 5)\cdot
    (1 \; 2)(3 \; 5)(4 \; 6)
    =(1 \; 2)
\end{align*}
and 
\begin{align*}
    (1\; 2\; 3\; 4\; 5\; 6) &= (1 \; 2)(2 \; 3)(3 \; 4)(4 \; 5)(5 \; 6) \\
    &\overset{\xi}{\mapsto} 
    (2\; 6)(3\; 5\; 4)  =
    (2\; 3)(3 \; 4)(4\; 5)(5\; 6)(4 \; 5)(3 \; 4)(2\; 3)(4 \; 5)(3 \; 4) \\
    &\overset{\xi}{\mapsto}
    (1\; 2\; 3\; 4\; 5\; 6). 
\end{align*}

\begin{lem}
    \label{lem:psi_xi_is_inn}
    A group automorphism $\psi\circ \xi^{-1}$ is an inner automorphism of $\Sf_6$ for any $\psi \in \auto(\Sf_6) \setminus \inn(\Sf_6)$. 
    In other words, for any outer automorphism $\psi$, there exists $g \in \Sf_6$ such that $\psi= (\cdot )^g \circ \xi$.
\end{lem}
\begin{proof}
    It is known that any outer automorphism of $\Sf_6$
    permutes its conjugacy classes like 
$\mathcal{C}_{(6)}\leftrightarrow \mathcal{C}_{(3,2,1)}$, $\mathcal{C}_{(3,3)}\leftrightarrow \mathcal{C}_{(3,1^3)}$, and 
$\mathcal{C}_{(2^3)}\leftrightarrow \mathcal{C}_{(2,1^4)}$ 
    (cf.~\cite{Wildon2008}). Thus, $\mathcal{C}_{(2,1^4)}$ is fixed by $\psi\circ \xi^{-1}$. 
    Therefore, $\psi\circ \xi^{-1}$ should be an inner automorphism of $\Sf_6$. 
\end{proof}

From Lemma~\ref{lem:psi_xi_is_inn}, every element of $\auto(\Sf_6)$ is written as $(\cdot)^g \circ \xi^{\varepsilon}$, 
where $g \in \Sf_6$ and $\varepsilon\in \{0,1\}$. 
Remark that when $\varepsilon =0$ (resp. $1$), $(\cdot)^g \circ \xi^{\varepsilon}$ is an inner (resp.~outer) automorphism. 
Since $\xi \circ (\cdot )^g = (\cdot )^{\xi(g)}\circ \xi$ for $g \in \Sf_6$, a semidirect product in $\auto(\Sf_6)$ is determined by 
$
((\cdot)^{g_1} \circ \xi^{\varepsilon_1})
    \circ
    ((\cdot)^{g_2} \circ \xi^{\varepsilon_2})
    :=
    (\cdot)^{g_1 \cdot \xi^{\varepsilon_1}(g_2)}
    \circ
    \xi^{\varepsilon_1+\varepsilon_2}.
$

Next, we give the conjugacy classes of $\auto(\Sf_6)$. 
For a partition $\lambda$ of $n=6$, let $\Ic_\lambda$ be the set of inner automorphisms $(\cdot)^\pi$ whose shape is $\lambda$. 
Then, for any $\psi \in \auto(\Sf_6) \setminus \inn(\Sf_6)$, we see that $$\psi^2 \in \Ic_{(5,1)} \cup \Ic_{(4,2)} \cup \Ic_{(2^2,1^2)} \cup \Ic_{(1^6)}.$$ 
In fact, given $\psi=(\cdot)^g \circ \xi \in \auto(\Sf_6) \setminus \inn(\Sf_6)$, since $\psi^2=(\cdot)^{g \cdot \xi(g)}$ and $g$ and $\xi(g)$ should have the same parity, 
we see that $g\cdot \xi(g) \in \Af_6$. Here, it follows from the result of Lam--Leep~\cite{Lam1993} that 
the order of an outer automorphism of $\Sf_6$ is $10$, $8$, $4$ or $2$, so the order of $\psi^2$ should be $5$, $4$, $2$ or $1$. 
This implies that $\psi^2$ coincides with $(\cdot )^g$, where $g$ has the shape $(5,1)$, $(4,2)$, $(2^2,1^2)$ or $(1^6)$, as required. 

For $\lambda=(5,1),(2^2,1^2),(1^6)$, let $\Oc_\lambda$ be the set of outer automorphisms $\psi$ with $\psi^2\in \Ic_\lambda$. Also, let 
\begin{align*}
    \Oc_{(4,2)}^{\mathrm{E}}
    &:=
    \{(\cdot)^g\circ \xi : \text{$((\cdot)^g\circ \xi)^2\in \Ic_{(4,2)}$ and $g$ is even}\}, \text{ and }\\
    \Oc_{(4,2)}^{\mathrm{O}}
    &:=
    \{(\cdot)^g\circ \xi : \text{$((\cdot)^g\circ \xi)^2\in \Ic_{(4,2)}$ and $g$ is odd}\}.
\end{align*}

Then the conjugacy classes of $\auto(\Sf_6)$ are the following thirteen classes; 
eight of them are inner automorphisms and five of them are outer automorphisms (cf. \cite{Lam1993}): 
    \begin{align*}
    &\inn(\Sf_6) : \; \Ic_{(6)}\cup\Ic_{(3,2,1)}, \ \Ic_{(5,1)}, \ \Ic_{(4,2)}, \ \Ic_{(4,1,1)}, \ \Ic_{(3,3)}\cup \Ic_{(3,1^3)}, \
                      \Ic_{(2^3)}\cup \Ic_{(2,1^4)}, \ \Ic_{(2,2,1,1)}, \ \Ic_{(1^6)}; \\
    &\auto(\Sf_6) \setminus \inn(\Sf_6): \; \Oc_{(5,1)}, \ \Oc_{(4,2)}^{\mathrm{E}}, \ \Oc_{(4,2)}^{\mathrm{O}}, \ \Oc_{(2^2,1^2)}, \ \Oc_{(1^6)}. 
    \end{align*}

Towards the classification of $\Qc(\Sf_6)$, we give the structure of $\inn(Q(\Sf_6,\psi))$ for $\psi\in \Oc_{(4,2)}^{\mathrm{E}}$ and $\psi\in \Oc_{(4,2)}^{\mathrm{O}}$. 
Let $\eta_0:= (\cdot)^{(2\; 5 \; 6 \; 4 \; 3)}\circ \xi$ and $\eta_1:= (\cdot)^{(1 \; 5 \; 6 \; 4)}\circ \xi$.
For $k=0,1$, it can be verified by hand that 
\begin{equation}
    \label{eq:eta^2}
    \fix(\eta_k,\Sf_6)  =  \langle (1 \; 2 \; 3 \; 4)(5 \; 6)\rangle
    \ 
    \text{and}
    \ 
    \eta_k^2=(\cdot)^{(1 \; 2 \; 3 \; 4)(5 \; 6)}.
\end{equation}
So we have $\ord(\eta_k)=8$. In particular, we have $\eta_0\in \Oc_{(4,2)}^{\mathrm{E}}$ and $\eta_1\in \Oc_{(4,2)}^{\mathrm{O}}$. 
\begin{lem}
    \label{lem:eta_to_A6}
    For $k=0,1$, $\eta_k$ acts on the conjugacy classes of $\Af_6$ as follows: 
    \begin{enumerate}[label=$($\alph*$)$]
        \item $\eta_0$ fixes
        $e^{\Af_6}$, $(1\; 2)(3\; 4)^{\Af_6}$,
        $(1\; 2\; 3\; 4)(5\; 6)^{\Af_6}$,
        $(1\; 2\; 3\; 4\; 5)^{\Af_6}$
        and
        $(1\; 2\; 3\; 4\; 6)^{\Af_6}$,
        and swaps $(1\; 2\; 3)^{\Af_6}\leftrightarrow (1\; 2\; 3)(4\; 5\; 6)^{\Af_6}$; 
        \item $\eta_1$ fixes 
        $e^{\Af_6}$, $(1\; 2)(3\; 4)^{\Af_6}$ and
        $(1\; 2\; 3\; 4)(5\; 6)^{\Af_6}$,
        and swaps $(1\; 2\; 3)^{\Af_6}\leftrightarrow (1\; 2\; 3)(4\; 5\; 6)^{\Af_6}$
        and $(1\; 2\; 3\; 4\; 5)^{\Af_6}\leftrightarrow (1\; 2\; 3\; 4\; 6)^{\Af_6}$.
    \end{enumerate}
\end{lem}
On the conjugacy classes of $\Af_6$, see, e.g., \cite[Chapter~11]{Scott}. 
\begin{proof}
    We can easily check that $\eta_k$ fixes $e^{\Af_6}$, $(1\; 2)(3\; 4)^{\Af_6}$ and $(1\; 2\; 3\; 4)(5\; 6)^{\Af_6}$, 
    and  swaps $(1\; 2\; 3)^{\Af_6}\leftrightarrow (1\; 2\; 3)(4\; 5\; 6)^{\Af_6}$. 

    (a)
    We can verify $\eta_0 ((1\; 2\; 3\; 4\; 5))=(1\; 2\; 3\; 4\; 5)^{(1\; 4)(3\; 6)}$
    and $\eta_0 ((1\; 2\; 3\; 4\; 6))=(1\; 2\; 3\; 4\; 6)^{(1\; 5)(2\; 4\; 6\; 3)}$.
    Hence $\eta_0$ fixes $(1\; 2\; 3\; 4\; 5)^{\Af_6}$ and $(1\; 2\; 3\; 4\; 6)^{\Af_6}$. 

    (b)
    We can verify
    $\eta_1((1\; 2\; 3\; 4\; 5))=(1\; 2\; 3\; 4\; 6)^{(1\; 4)(2\; 5)}$.
    Hence $\eta_1$ swaps $(1\; 2\; 3\; 4\; 5)^{\Af_6}\leftrightarrow (1\; 2\; 3\; 4\; 6)^{\Af_6}$.
\end{proof}

For $k=0,1$, let $Q_k:=Q(\Sf_6,\eta_k)$. 

\begin{lem}\label{lem:inn(outer8)}
For $k=0,1$, we have $$\inn (Q_k)\cong_G \Af_6 \rtimes_{\varphi_k} C_8,$$ 
where $\Af_6 \rtimes_{\varphi_k} C_8$ is the semidirect product of $\Af_6$ and $C_8$ with
$\varphi_k\colon C_8 \to \auto(\Af_6)$, $\overline{i} \mapsto \eta_k^i$, that is,
\[
    (g,\overline{i})\cdot(h,\overline{j}):=(g \eta_k^{i} (h), \overline{i+j})
\]
for $(g,\overline{i}),(h,\overline{j}) \in \Af_6 \times C_8$.
\end{lem}
\begin{proof}
    For $k=0,1$ and $x\in \Sf_6$, define $s^k_x\in \inn(Q_k)$ by $s^k_x (y):=x\eta_k (x^{-1} y)$.
    Consider $s^k_{x_1}\circ s^k_{x_2}\circ \cdots \circ s^k_{x_i}$ for any $x_1,x_2,\ldots ,x_i\in \Sf_6$. Then we have
    \begin{align*}
        s^k_{x_1}\circ s^k_{x_2}\circ \cdots \circ s^k_{x_i} (y)
        &=
        x_1 \eta_k (x_1^{-1}) \cdot \eta_k (x_2 \eta_k (x_2^{-1}))
        \cdot \cdots \cdot \eta_k^{i-1} (x_i \eta_k (x_i^{-1})) \cdot \eta_k^i (y).
    \end{align*}
    Put $\sigma:=x_1 \eta_k (x_1^{-1}) \cdot \eta_k (x_2 \eta_k (x_2^{-1})) \cdot \cdots \cdot \eta_k^{i-1} (x_i \eta_k (x_i^{-1}))$. 
    Since $x$ and $\eta_k^j(x^{-1})$ have the same parity for any $x \in \Sf_6$ and $j \in \ZZ$, we know that $\sigma \in \Af_6$. 
    Hence, we can define the map $\Phi_k \colon \inn (Q_k) \to \Af_6 \rtimes_{\varphi_k} C_8$ by setting 
    \[
        s^k_{x_1}\circ s^k_{x_2}\circ \cdots \circ s^k_{x_i} \mapsto (\sigma, \overline{i}).
    \]

    \noindent
    {\bf Well-definedness}: 
    Let us write $f \in \inn (Q_k)$ in two ways: 
    \[
        f=s^k_{x_1}\circ s^k_{x_2}\circ \cdots \circ s^k_{x_i} 
        =s^k_{y_1}\circ s^k_{y_2}\circ \cdots \circ s^k_{y_j}.
    \]
    Then there are $\sigma_1,\sigma_2\in \Af_6$ such that for any $z\in \Sf_6$,
    \begin{align*}
        s^k_{x_1}\circ s^k_{x_2}\circ \cdots \circ s^k_{x_i} (z) = \sigma_1 \eta_k^i(z) \;\text{ and }\; 
        s^k_{y_1}\circ s^k_{y_2}\circ \cdots \circ s^k_{y_j} (z) = \sigma_2 \eta_k^j(z), 
    \end{align*}
    respectively. Thus, we have $\sigma_1 \eta_k^i(z) = \sigma_2 \eta_k^j(z)$ for any $z\in \Sf_6$. 
    By substituting $z=e$, we see $\sigma_1=\sigma_2$. Hence, $\eta_k^i = \eta_k^j \ \Longleftrightarrow \ \eta_k^{i-j} (z) = \id$. 
    Since $\ord (\eta_k)=8$, we obtain that $i \equiv j \pmod 8$. 

    \noindent
    {\bf Homomorphism}: 
    For $f_1,f_2\in \inn(Q_k)$, write $f_1(x)=g \eta_k^i(x)$ and $f_2(x)=h \eta_k^j(x)$ with $g,h \in \Af_6$. Then we have
    \begin{align*}
        \Phi_k(f_1\circ f_2)&=\Phi_k(g \eta_k^i(h \eta_k^j(\cdot )))=\Phi_k(g \eta_k^i(h) (\eta_k^{i+j}(\cdot )))
        =(g \eta_k^i(h) , \overline{i+j})=(g,\overline{i})\cdot(h,\overline{j})\\
        &=\Phi_k(f_1) \Phi_k(f_2).
    \end{align*}

    \noindent
    {\bf Injectivity}: 
    The injectivity of $\Phi_k$ directly follows from the proof of well-definedness. 

    \noindent
    {\bf Surjectivity}: We can check that
    \begin{align*}
        s^0_{(4\; 5\; 6)}\circ s^0_{(2\; 6 \; 4\; 5)}(e) &=
        s^1_{(2\; 3\; 5\; 6)}\circ s^1_e(e)=
        (1\; 2\; 3),\\
        s^0_{(1\; 4)(3\; 6)}\circ s^0_e(e) &=
        s^1_{(5 \; 6)}\circ s^1_{(2\; 4\; 5)(3\; 6)}(e)=
        (1\; 2\; 4),\\
        s^0_{(5\; 6)}\circ s^0_{(1 \;6 \; 3 \;4)}(e) &=
        s^1_{(2\; 5)(3\; 6)}\circ s^1_e(e)=
        (1\; 2\; 5),\\
        s^0_{(1\; 2)(3 \; 4)}\circ s^0_e(e) &=
        s^1_{(2\; 6\; 5\; 3)}\circ s^1_e(e) =
        (1\; 2\; 6).
    \end{align*}
    This implies that for any $k=0,1$ and $j=3,4,5,6$, there exist $x_1,x_2 \in \Sf_6$
    such that $\Phi_k(s^k_{x_1} \circ s^k_{x_2} \circ (s^k_e)^6)=((1\; 2\; j),\overline{0})$. 
    Since $(g,\overline{0})\cdot (h,\overline{0})=(gh, \overline{0})$
    and Lemma~\ref{lem:gene_S_A} (b), it follows that for any $g\in \Af_6$, 
    there exists $f_g \in \inn(Q_k)$ such that $\Phi_k(f_g)=(g,\overline{0})$. 
    Therefore for any $(g,\overline{i})\in \Af_6 \rtimes_{\varphi_k} C_8$, 
    one has $\Phi_k(f_g\circ (s^k_{e})^i)=(g,\overline{i})$. 
\end{proof}

In order to prove Proposition~\ref{prop:Q_1vsQ_2}, 
we calculate the conjugation in $\Af_6 \rtimes_{\varphi_k} C_8$. 
For $k=0,1$ and $(g,\overline{i}),(h,\overline{j})\in \Af_6 \rtimes_{\varphi_k} C_8$, we have
\begin{align}
    \label{eq:conj_AxZ/8Z}
     (g,\overline{i})^{(h,\overline{j})}
    =
    (h\eta_k^j(g),\overline{i+j})\cdot (\eta_k^{-j} (h^{-1}), \overline{-j})
    =
    (h\eta_k^j(g) \eta_k^i(h^{-1}) , \overline{i}).
\end{align}

\begin{prop}
    \label{prop:Q_1vsQ_2}
    Work with the notation as above.
    \begin{enumerate}[label=$($\alph*$)$]
        \item The centralizer
        $C_{\Af_6 \rtimes_{\varphi_0} C_8}(((1\; 2\; 3\; 4\; 5),\overline{0}))$
	has $40$ elements. 
        \item $\Af_6 \rtimes_{\varphi_1} C_8$ has no element $(g,\overline{i})$
        such that $|C_{\Af_6 \rtimes_{\varphi_1} C_8}((g,\overline{i}))|=40$.
    \end{enumerate}
    In particular, $Q_0 \not\cong_Q Q_1$.
\end{prop}
\begin{proof}    
    (a)
    In $\Af_6 \rtimes_{\varphi_0} C_8$,
    we have
    \begin{align*}
        ((1\; 2\;3\; 4\; 5),\overline{0})^{(h,\overline{j})}=((1\; 2\;3\; 4\; 5),\overline{0}) 
        &\Longleftrightarrow \ 
        (h\eta_0^j((1\; 2\;3\; 4\; 5)) h^{-1} , \overline{0})=((1\; 2\;3\; 4\; 5),\overline{0})\\
        &\Longleftrightarrow \ \eta_0^j((1\; 2\;3\; 4\; 5))^h = (1\; 2\;3\; 4\; 5).
    \end{align*}
    Fix $\overline{j} \in C_8$. We see from Lemma~\ref{lem:eta_to_A6} that $\eta_0^j$ fixes $(1\; 2\; 3\; 4\; 5)^{\Af_6}$, 
    so there exists $a_j\in \Af_6$ such that 
    $\eta_0^j((1\; 2\;3\; 4\; 5))=(1\; 2\;3\; 4\; 5)^{a_j}$. Thus we obtain that 
    $h a_j\in C_{\Af_6}((1\; 2\; 3\; 4\; 5))=\{(1 \; 2 \; 3 \; 4 \; 5)^\ell : \ell \in C_5\}$. 
    Therefore, 
    \[
        C_{\Af_6 \rtimes_{\varphi_0} C_8}(((1\; 2\; 3\; 4\; 5),\overline{0}))
        =\left\{
            ((1\; 2\; 3\; 4\; 5)^\ell \cdot a_j^{-1}, \overline{j})
        : \overline{j}\in C_8,\  \overline{\ell}\in C_5
        \right\}
    \]
    and this implies that $|C_{\Af_6 \rtimes_{\varphi_0} C_8}(((1\; 2\; 3\; 4\; 5),\overline{0}))|=40$. 

    (b)
    Assume that there exists an element $(g,\overline{i}) \in \Af_6 \rtimes_{\varphi_1} C_8$ with $|C_{\Af_6 \rtimes_{\varphi_1} C_8}((g,\overline{i}))|=40$. 
    Put $X:=C_{\Af_6 \rtimes_{\varphi_1} C_8}((g,\overline{i}))$.
    By the Cauchy's theorem (special case of Sylow's Theorem, cf.~\cite[Theorem~4.2]{Rotman}),
    $X$ has an element $\alpha$ of order $5$.
    Since $\alpha^5=(h',\overline{5j})=(e,\overline{0})$, we have $5j\equiv 0 \pmod{8}$, that is, $\overline{j}=\overline{0}$. 
    We put $\alpha=(h_\alpha,\overline{0})$.

    \noindent
    (i) Let $i$ be odd. 
    Since $\eta_1^i=(\cdot)^{z_i}\circ \eta_1$, where $z_i=((1 \; 2 \; 3 \; 4)(5 \; 6))^{(i-1)/2}$, we have 
    \begin{align*}
        (g,\overline{i})^\alpha = (g,\overline{i}) \Longleftrightarrow \ h_\alpha g \eta_1^i(h_\alpha^{-1}) = g 
        \Longleftrightarrow \ \eta_1(h_\alpha^{-1}) = (gz_i)^{-1} h_\alpha^{-1} (gz_i). 
    \end{align*}
    This means that $\eta_1$ fixes $(h_\alpha^{-1})^{\Af_6}$. By Lemma~\ref{lem:eta_to_A6},
    $h_\alpha^{-1}\in e^{\Af_6}\cup (1\; 2)(3\; 4)^{\Af_6}\cup (1\; 2\; 3\; 4)(5\; 6)^{\Af_6}$ holds.
    This is a contradiction since the order of $h_\alpha$ is $5$.

    \noindent
    (ii) Let $i$ be even. 
    Consider the projection $\Pi\colon X \to C_8$.
    Firstly, we show that $\ker \Pi = \langle \alpha \rangle$. 
    By $\alpha=(h_\alpha,\overline{0})$, we have $\langle \alpha \rangle \subset \ker \Pi$.
    On the other hand, since $\eta_1^i=(\cdot)^{z_i}$, where $z_i=((1 \; 2 \; 3 \; 4)(5 \; 6))^{i/2}$. 
    In fact, by \eqref{eq:eta^2}, we see that $z_i \in \fix(\eta_1,\Sf_6)$. Then, 
    \begin{align*}
        (g,\overline{i})^{(h,\overline{0})} = (g,\overline{i}) \ \Longleftrightarrow \
        h g z_i h^{-1} z_i^{-1}= g \ \Longleftrightarrow \ 
        h (g z_i) =  (g z_i) h
        \label{eq:comm}
    \end{align*}
    for any $(h,\overline{0})\in \ker \Pi$. This implies that $\ker \Pi \subset C_{\Af_6}(g z_i)\times \{\overline{0}\}$. 
    Hence, $\langle \alpha \rangle \subset \ker \Pi \subset C_{\Af_6}(g z_i)\times \{\overline{0}\}$. 
    Since $\langle h_\alpha \rangle$ is a subgroup of $C_{\Af_6}(g z_i)$, $|C_{\Af_6}(g z_i)|$ is divisible by $5$. 
    Here, we know that $|C_{\Af_6}(h')|$ is divisible by $5$ if and only if $h'=e$ or $h'$ is $5$-cycle (see, e.g., \cite[Chapter~11]{Scott}). 
    If $g z_i=e$, then we can check by using \eqref{eq:conj_AxZ/8Z} that for any $h' \in \Af_6$, 
    \begin{align*}
    (g,\overline{i})^{(h',\overline{0})}=(h'gz_ih'^{-1}z_i^{-1},\overline{i})=(g,\overline{i}) \ \Longleftrightarrow \ (h',\overline{0}) \in X, 
    \end{align*}
    a contradiction to $|X|=40$. 
    Thus, $g z_i$ must be a $5$-cycle. Then $|C_{\Af_6}(g z_i)|=5$, so we have $\langle \alpha \rangle=\ker \Pi=C_{\Af_6}(gz_i)\times \{\overline{0}\}$. 
    
    Since $\Pi(X) \cong_G X/\ker \Pi $ and $|\Pi(X)|=40/5=8$, we see that $\Pi$ is surjective. 
    So we can find $(h'',\overline{1}) \in X$. Using \eqref{eq:conj_AxZ/8Z}, it is verified that $(h'',\overline{1})$ satisfies 
    \begin{align*}
        (g,\overline{i})^{(h'',\overline{1})} = (g,\overline{i}) \ \Longleftrightarrow \ h''\eta_1(g)z_ih''^{-1}z_i^{-1}=g 
        \ \Longleftrightarrow \ \eta_1(gz_i) = (gz_i)^{h''^{-1}}. 
    \end{align*}
    This is a contradiction since $g z_i$ is a $5$-cycle and $(g z_i)^{\Af_6}$ is not a fixed class by the action of $\eta_1$ by Lemma~\ref{lem:eta_to_A6}. 
\end{proof}


\subsection{Double cosets of $\Sf_n$ and the invariant of $Q(\pi)$}

Fix $\pi \in \Sf_n$. 
Let us introduce the other invariant of $Q(\pi)$. Let $K=C_{\Sf_n}(\pi)$ and let $$K_\mathrm{alt}=K \cap \Af_n.$$ 
Given $g \in \Sf_n$, let $K_\mathrm{alt}\cdot g \cdot K=\{h g k : h \in K_\mathrm{alt}, k \in K\}$. 
Let $K_\mathrm{alt} \setminus \Sf_n / K$ denote the double coset representatives with respect to $K_\mathrm{alt}$ and $K$. 

Similarly, for $\pi' \in \Sf_n$, let $K'=C_{\Sf_n}(\pi')$ and $K_\mathrm{alt}'=K' \cap \Af_n$. 
\begin{prop}\label{prop:double_coset}
Assume that $Q(\pi) \cong_Q Q(\pi')$. Then $|K_\mathrm{alt} \setminus \Sf_n / K|=|K_\mathrm{alt}' \setminus \Sf_n / K'|$. 
\end{prop}
Before proving this proposition, we prepare some materials. 

In general, let us consider a generalized Alexander quandle $Q=Q(G,\psi)$, where $G$ is a finite group and $\psi \in \auto(G)$. 
Given $h \in G$, let $L_h$ denote the map $L_h:G \rightarrow G$ defined by $L_h(g)=hg$ for $g \in Q$. 
Then we see that $L_h \in \auto(Q)$. In fact, $L_h$ is a bijection on $Q$ since $Q=G$ as a set, and we also have 
$$
L_h \circ s_x(y)=hx\psi(x^{-1}y)=hx\psi(x^{-1}h^{-1}hy)=s_{hx}(hy)=s_{L_h(x)} \circ L_h(y), 
$$
i.e., $L_h$ is a quandle homomorphism. Note that $L_h^{-1}=L_{h^{-1}}$. 

Let $\psi' \in \auto(G)$ and assume that $Q(G,\psi) \cong_Q Q(G,\psi')$. 
Then, without loss of generality, we may assume that there is a quandle automorphism $f : Q(G,\psi) \rightarrow Q(G,\psi')$ satisfying $f(e)=e$. 
In fact, $f':=L_{f(e)}^{-1} \circ f$ is a quandle automorphism between $Q(G,\psi)$ and $Q(G,\psi')$ with $f'(e)=e$. 

We prepare the following lemmas for the proof of Proposition~\ref{prop:double_coset}. 
\begin{lem}\label{lem:111}
Work with the same notation as above. Let $K=\fix(\psi,G)$. Take any subgroup $H$ of $G$. Then, for $x_1,x_2 \in G$, we have 
$$H \cdot x_1 \cdot K=H \cdot x_2 \cdot K \;\Longleftrightarrow\; s_{x_2}=s_{hx_1} \text{ for some }h \in H.$$
\end{lem}
\begin{proof}
($\Rightarrow$): Suppose that $H \cdot x_1 \cdot K=H \cdot x_2 \cdot K$. 
Then there exist $h \in H$ and $k \in K$ such that $x_2=hx_1k$. Thus, for any $y \in Q(G,\psi)$, we see that 
\begin{align*}
s_{x_2}(y)&=x_2\psi( x_2^{-1}y )=hx_1k \psi( (hx_1k)^{-1}y )=hx_1 k \psi( k^{-1} (hx_1)^{-1} y ) = (hx_1) \psi ((hx_1)^{-1} y ) \\
&= s_{hx_1}(y). 
\end{align*}
($\Leftarrow$): Suppose that $s_{x_2}=s_{hx_1}$. Then for any $y \in Q(G,\psi)$, we have 
$$s_{x_2}(y)=x_{hx_1}(y) \;\Longleftrightarrow\; 
x_2 \psi (x_2^{-1} y )=hx_1 \psi ((hx_1)^{-1}y ) \;\Longleftrightarrow\; (hx_1)^{-1}x_2 = \psi ((hx_1)^{-1}x_2).$$
Hence, $x_1^{-1}h^{-1}x_2 \in \fix(\psi,G)=K$. Namely, there exists $k \in K$ such that $x_1^{-1}h^{-1}x_2=k$, i.e., $x_2 = h x_1 k$, as required. 
\end{proof}

\begin{lem}\label{lem:222}
Let $f:Q(\pi) \rightarrow Q(\pi')$ be a quandle automorphism with $f(e)=e$. 
For each $h \in K_\mathrm{alt}$, there exists $h' \in K_\mathrm{alt}'$ such that $f \circ L_h \circ f^{-1}=L_{h'}$. 
\end{lem}
\begin{proof}
\noindent
{\bf (The first step)}: Since $h \in K_\mathrm{alt} = K \cap \Af_n$, we have $h \in \Af_n$. 
Then there exist $x_1,\ldots,x_\ell \in \Sf_n$ such that $L_h=s_{x_1} \circ \cdots \circ s_{x_\ell}=\pi^{x_1} \cdots \pi^{x_\ell}( \cdot )$. 
(See the proof of Lemma~\ref{lem:inn}.) Thus, we see that 
\begin{align*}
f \circ L_h \circ f^{-1} &= f \circ s_{x_1} \circ \cdots \circ s_{x_\ell} \circ f^{-1} 
=s_{f(x_1)}' \circ \cdots \circ s_{f(x_\ell)}' =\pi'^{f(x_1)} \cdots \pi'^{f(x_\ell)}( \cdot ). 
\end{align*}
Let 
\begin{equation}
    \label{eq:hhh}
    h'=\pi'^{f(x_1)} \cdots \pi'^{f(x_\ell)} \in \Sf_n.
\end{equation}
Then we have $f \circ L_h \circ f^{-1}=L_{h'}$. Hence, we may prove that $h' \in K_\mathrm{alt}'$.

\noindent
{\bf (The second step)}: If $\pi'$ is an even permutation, it is clear that $h' \in \Af_n$ by \eqref{eq:hhh}. 
Even if $\pi'$ is odd, then $\ord(\pi')$ is even. Thus, we have that $\ell$ should be even, so $h' \in \Af_n$ holds. 

\noindent
{\bf (The third step)}: Our remaining task is to show that $h' \in K'$, i.e., $\psi'(h')=h'$. 
By the first step, we have $f \circ L_h \circ f^{-1}=L_{h'}$. Moreover, we also have that 
$s_e \circ L_h(x)=\psi(hx)=\psi(h)\psi(x)=L_{\psi(h)} \circ s_e(x)$ 
for any $x \in \Sf_n$, i.e., $s_e \circ L_h = L_{\psi(h)} \circ s_e$. Thus, 
\begin{align*}
s_e' \circ L_{h'} &= s_e' \circ f \circ L_h \circ f^{-1} = f \circ s_{f^{-1}(e)} \circ L_h \circ f^{-1} 
=f \circ s_e \circ L_h \circ f^{-1} = f \circ L_{\psi(h)} \circ s_e \circ f^{-1} \\
&=f \circ L_h \circ f^{-1} \circ s_{f(e)}' = L_{h'} \circ s_e'. 
\end{align*}
Hence, $\psi'(h')=s_e' \circ L_{h'}(e)=L_{h'} \circ s_e'(e)=h'$, as required. 
\end{proof}

\begin{proof}[Proof of Proposition~\ref{prop:double_coset}]
Let $f:Q(\pi) \rightarrow Q(\pi')$ be a quandle automorphism. Without loss of generality, we may assume that $f(e)=e$. 
For the statement, it is enough to show that for $x_1,x_2 \in \Sf_n$, we have 
$$K_\mathrm{alt} \cdot x_1 \cdot K=K_\mathrm{alt} \cdot x_2 \cdot K \; \Longleftrightarrow \; 
K_\mathrm{alt}' \cdot f(x_1) \cdot K'=K_\mathrm{alt}' \cdot f(x_2) \cdot K'.$$ 

First, we have 
$$K_\mathrm{alt}\cdot x_1 \cdot K = K_\mathrm{alt} \cdot x_2 \cdot K \;\Longleftrightarrow\; s_{x_2} = s_{hx_1} \text{ for some }h \in K_\mathrm{alt}$$
by Lemma~\ref{lem:111}. Moreover, we have 
\begin{align*}
s_{x_2} = s_{hx_1} \;&\Longleftrightarrow\; f \circ s_{x_2} \circ f^{-1} = f \circ s_{hx_1} \circ f^{-1} \;\Longleftrightarrow\; 
s_{f(x_2)}'=f \circ L_h \circ s_{x_1} \circ L_h^{-1} \circ f^{-1} \\
&\Longleftrightarrow\; s_{f(x_2)}'=(f \circ L_h \circ f^{-1}) \circ s_{f(x_1)}' \circ (f \circ L_h \circ f^{-1})^{-1}. 
\end{align*}
Now, by Lemma~\ref{lem:222}, we see that $f \circ L_h \circ f^{-1}=L_{h'}$ for some $h' \in K_\mathrm{alt}'$. Hence, 
\begin{align*}
s_{x_2} = s_{hx_1} \;&\Longleftrightarrow\; s_{f(x_2)}'=L_{h'} \circ s_{f(x_1)}' \circ L_{h'}^{-1} \;\Longleftrightarrow\; s_{f(x_2)}'=s_{h'f(x_1)}' \\
&\Longleftrightarrow \; K_\mathrm{alt}' \cdot f(x_1) \cdot K'=K_\mathrm{alt}' \cdot f(x_2) \cdot K' \;\;\text{by Lemma~\ref{lem:111}}, 
\end{align*}
as required. 
\end{proof}

\subsection{Determining the structure of $\Qc(\Sf_n)$}

We are now ready to discuss Problem \ref{toi} for the symmetric group $\Sf_n$ with $n \geq 3$. 
In order to distinguish $Q(\Sf_n,\psi)$'s for $\psi \in \auto(\Sf_n)$, we employ the following strategy: 
\begin{itemize}
\item[(i)] Compute $\ord(\psi)$ and $|\fix(\psi,\Sf_n)|$ and check if one of those two invariants is different. 
(Recall that $\ord(\psi)=\ord(\pi)$ and $\fix(\psi,\Sf_n)=C_{\Sf_n}(\pi)$ if $\psi=(\cdot)^\pi$.) 
\item[(ii)] When $(\ord(\psi), |\fix(\psi,\Sf_n)|)=(\ord(\psi'), |\fix(\psi',\Sf_n)|)$ for some $\psi,\psi' \in \auto(\Sf_n)$ which are not conjugate, 
compare $\inn(Q(\Sf_n,\psi))$ and $\inn(Q(\Sf_n,\psi'))$. 
If $\psi=(\cdot)^\pi$ and $\psi'=(\cdot)^{\pi'}$, then it is enough to compare their parities of $\pi$ and $\pi'$ by Proposition~\ref{prop:innQ}. 
\item[(iii)] When those also coincide, compute $\psi^i$ and $\psi'^i$ for some positive integer $i$, 
and compare $Q(\Sf_n,\psi^i)$ and $Q(\Sf_n,\pi'^i)$ in the ways (i) and (ii) above. 
\item[(iv)] When those still agree and $\psi=(\cdot)^\pi$ and $\psi'=(\cdot)^{\pi'}$, apply Proposition~\ref{prop:double_coset}. 
Namely, compute $K=C_{\Sf_n}(\pi), K_\mathrm{alt}=K \cap \Af_n$, $K'=C_{\Sf_n}(\pi')$ and $K'_\mathrm{alt}=K' \cap \Af_n$, 
and compare $|K_\mathrm{alt} \setminus \Sf_n / K|$ and $|K_\mathrm{alt}' \setminus \Sf_n / K'|$. 
\end{itemize}

The following tables show that the different conjugacy classes of $\Sf_n$ determine the different generalized Alexander quandles. 

\bigskip

\noindent
\underline{$n=3,4,5$}

For the cases $n=3,4,5$, we can distinguish $Q(\pi)$ by considering $\ord(\pi)$ and $|C_{\Sf_n}(\pi)|$. 
See Tables~\ref{tab:S_3}, \ref{tab:S_4} and \ref{tab:S_5}. Hence, the strategy (i) is enough. 

\begin{table}[h]
\begin{minipage}[t]{.48\textwidth}
\centering
\small
\begin{tabular}{r|ccc}
Shape of $\pi$        & $(3)$ &$(2,1)$ &$(1,1,1)$ \\
\toprule
$\ord(\pi)$           & $3$   &$2$     &$1$ \\
$|C_{\Sf_3}(\pi)|$        & $3$   &$2$     &$6$ \\
\bottomrule
\end{tabular}

\bigskip

\caption{Conjugacy classes for $\Sf_3$ and the invariants}
\label{tab:S_3}
\end{minipage}
\begin{minipage}[t]{.5\textwidth}
\centering
\small
\begin{tabular}{r|ccccc}
Shape of $\pi$            & $(4)$ &$(3,1)$ &$(2,2)$ &$(2,1,1)$ &$(1,1,1,1)$ \\
\toprule
$\ord(\pi)$               & $4$   &$3$     &$2$     &$2$       &$1$ \\
$|C_{\Sf_4}(\pi)|$            & $4$   &$3$     &$8$     &$4$       &$24$ \\
\bottomrule
\end{tabular}

\bigskip

\caption{Conjugacy classes for $\Sf_4$ and the invariants}
\label{tab:S_4}
\end{minipage}

\bigskip

\centering
\small
\begin{tabular}{r|ccccccc}
Shape of $\pi$            & $(5)$ &$(4,1)$ &$(3,2)$ &$(3,1,1)$ &$(2,2,1)$ &$(2,1,1,1)$ &$(1,1,1,1,1)$ \\
\toprule
$\ord(\pi)$               & $5$   &$4$     &$6$     &$3$       &$2$       &$2$         &$1$ \\
$|C_{\Sf_5}(\pi)|$            & $5$   &$4$     &$6$     &$6$       &$8$       &$12$        &$120$ \\
\bottomrule
\end{tabular}

\bigskip

\caption{Conjugacy classes for $\Sf_5$ and the invariants}
\label{tab:S_5}
\end{table}

\bigskip

\noindent
\underline{$n=6$}

As we see in Tables~\ref{tab:S_6_inn} and \ref{tab:S_6_out} below, by using the strategies (i) and (ii), 
we can distinguish $Q(\Sf_6,\psi)$'s for $\psi \in \auto(\Sf_6)$ 
except for the cases of the pair $\Oc_{(4,2)}^{\mathrm{E}}$ and $\Oc_{(4,2)}^{\mathrm{O}}$, 
but we can distinguish these quandles by Proposition~\ref{prop:Q_1vsQ_2}. 

\begin{table}[H]
    \centering
    \small
    \begin{tabular}{r|cccccccc}
        Conjugacy classes   & $\Ic_{(6)}\cup\Ic_{(3,2,1)}$ & $\Ic_{(5,1)}$ & $\Ic_{(4,2)}$ & $\Ic_{(4,1,1)}$ & $\Ic_{(3,3)}\cup \Ic_{(3,1^3)}$ & $\Ic_{(2^3)}\cup \Ic_{(2,1^4)}$ & $\Ic_{(2,2,1,1)}$ & $\Ic_{(1^6)}$ \\
        \toprule
        $\ord(\pi)$         & $6$                          & $5$           & $4$           & $4$             & $3$                             & $2$                             & $2$               & $1$           \\
        $|\fix((\cdot)^\pi, \Sf_6)|$  & $6$                          & $5$           & $8$           & $8$             & $18$                            & $48$                            & $16$              & $6!$          \\
        Parity of $\pi$     &                              &               & even          & odd             &                                 &                                 &                   &        \\
        \bottomrule
    \end{tabular}

\bigskip

\caption{Conjugacy classes for $\inn(\Sf_6)$ and the invariants}
\label{tab:S_6_inn}
\end{table}
\begin{table}[H]
    \begin{tabular}{r|ccccc}
        Conjugacy classes    & $\Oc_{(5,1)}$ & $\Oc_{(4,2)}^{\mathrm{E}}$ & $\Oc_{(4,2)}^{\mathrm{O}}$ & $\Oc_{(2^2,1^2)}$ & $\Oc_{(1^6)}$ \\
        \toprule
        $\ord(\psi)$         & $10$          & $8$                        & $8$                        & $4$               & $2$           \\
        $|\fix(\psi,\Sf_6)|$ & $5$           & $4$                        & $4$                        & $4$               & $20$          \\
        \bottomrule
    \end{tabular}

\bigskip

\caption{Conjugacy classes for $\auto(\Sf_6) \setminus \inn(\Sf_6)$ and the invariants}
\label{tab:S_6_out}
\end{table}

\bigskip

\noindent
\underline{$n=7$}

According to Table~\ref{tab:S_7}, the strategies (i) and (ii) are enough. 

\begin{table}[H]
\centering
\small
\begin{tabular}{r|cccccccc}
Shape of $\pi$         & $(7)$ &$(6,1)$ &$(5,2)$ &$(4,3)$ &$(5,1^2)$ &$(4,2,1)$ &$(3^2,1)$ &$(3,2^2)$ \\ 
\toprule
$\ord(\pi)$            & $7$   &$6$     &$10$    &$12$    &$5$       &$4$       &$3$       &$6$       \\
$|C_{\Sf_7}(\pi)|$        & $7$   &$6$     &$10$    &$12$    &$10$      &$8$       &$18$      &$24$      \\
\bottomrule
\end{tabular}

\vspace{0.2cm}

\begin{tabular}{r|cccccccc}
Shape of $\pi$   &$(4,1^3)$ &$(3,2,1^2)$ &$(2^3,1)$ &$(2^2,1^3)$ &$(3,1^4)$ &$(2,1^5)$ &$(1^7)$ \\
\toprule
$\ord(\pi)$      &$4$       &$6$         &$2$       &$2$     &$3$           &$2$       &$1$ \\
$|C_{\Sf_7}(\pi)|$   &$24$      &$12$        &$48$      &$48$    &$72$       &$240$     &$7!$ \\
Parity               &          &            &odd       &even    &           &          &     \\
\bottomrule
\end{tabular}

\bigskip

\caption{Conjugacy classes for $\Sf_7$ and the invariants}\label{tab:S_7}
\end{table}

\bigskip

\noindent
\underline{$n=8$}

We see that we can distinguish $Q(\pi)$'s except for $(3^2,2)$ and $(3,2,1^3)$ (see Table~\ref{tab:S_8}). 
Now, we apply the strategy (iii) for those quandles. Let $\pi_1 \in \Sf_8$ (resp. $\pi_2 \in \Sf_8$) be of the shape $(3^2,2)$ (resp. $(3,2,1^3)$). 
Consider $Q_1:=Q(\pi_1)^{(2)}=Q(\pi_1^2)$ and $Q_2:=Q(\pi_2)^{(2)}=Q(\pi_2^2)$ (see Proposition~\ref{prop:i_bai} (b)). 
Since $\pi_1^2$ (resp. $\pi_2^2$) is of the shape $(3^2,1^2)$ (resp. $(3,1^5)$), we can see that $Q_1 \not\cong_Q Q_2$. 
Thus, we conclude that $Q(\pi_1) \not\cong Q(\pi_2)$ by Proposition~\ref{prop:i_bai} (a). 

Hence, we see that the quandles arising from different conjugacy classes are all different. 

\begin{table}[H]
\centering
\small
\begin{tabular}{r|ccccccccccccc}
Shape of $\pi$       & $(8)$ &$(7,1)$ &$(6,2)$ &$(6,1^2)$ &$(5,3)$ &$(4^2)$ &$(4,2^2)$ &$(5,2,1)$ \\ 
\toprule
$\ord(\pi)$      & $8$   &$7$     &$6$     &$6$       &$15$    &$4$     &$4$       &$10$      \\
$|C_{\Sf_8}(\pi)|$   & $8$   &$7$     &$12$    &$12$      &$15$    &$32$    &$32$      &$36$      \\
Parity               &       &        &even    &odd       &        &even    &odd       &          \\
\bottomrule
\end{tabular}
\end{table}
\begin{table}[H]
\begin{tabular}{r|ccccccccccc}
Shape of $\pi$      &$(4,3,1)$ &$(3^2,2)$ &$(3,2,1^3)$ &$(2^4)$ & $(5,1^3)$ &$(4,2,1^2)$ &$(3^2,1^2)$ \\
\toprule																
$\ord(\pi)$     &$12$      &$6$       &$6$         &$2$     & $5$       &$4$         &$3$         \\
$|C_{\Sf_8}(\pi)|$  &$10$      &$36$      &$36$        &$384$   & $30$      &$16$        &$36$        \\
Parity              &          &odd       &odd         &        &           &            &            \\
\bottomrule
\end{tabular}
\end{table}
\begin{table}[H]
\begin{tabular}{r|ccccccccccc}
Shape of $\pi$     &$(3,2^2,1)$ &$(4,1^4)$ &$(2^3,1^2)$ &$(3,1^5)$ &$(2^2,1^4)$ &$(2,1^6)$ &$(1^8)$ \\
\toprule																	
$\ord(\pi)$        &$6$         &$4$       &$2$         &$3$       &$2$         &$2$       &$1$ \\
$|C_{\Sf_8}(\pi)|$ &$24$        &$96$      &$24$        &$360$     &$96$        &$1440$    &$8!$ \\
\bottomrule
\end{tabular}

\bigskip

\caption{Conjugacy classes for $\Sf_8$ and the invariants}\label{tab:S_8}
\end{table}

\bigskip

\noindent
\underline{$n=9,11,12,16,19,20,23,28$}

In these cases, similar to the case $n=8$, we can distinguish $Q(\pi)$'s by the strategies (i), (ii) and (iii). 

\bigskip

\noindent
\underline{$n=10$}

In this case, we can distinguish $Q(\pi)$'s by the strategies (i), (ii) and (iii) except for $(4,2^3)$ and $(4,2,1^4)$. 
Let $\pi_1 \in \Sf_{10}$ (resp. $\pi_2 \in \Sf_{10}$) be of the shape $(4,2^3)$ (resp. $(4,2,1^4)$). 
In fact, we see the invariants of $Q(\pi_1)$ and $Q(\pi_2)$ as follows: 
\begin{align*}
\ord(\pi_1)=\ord(\pi_2)=4, \;\; |C_{\Sf_{10}}(\pi_1)|=|C_{\Sf_{10}}(\pi_2)|=192, 
\;\; \text{(parity of $\pi_1$)}=\text{(parity of $\pi_2$)}=\text{even}.
\end{align*}
Moreover, we see that $Q(\pi_1)^{(i)}=Q(\pi_1)$ (resp. $Q(\pi_2)^{(i)}=Q(\pi_2)$) for $i$ odd, 
$Q(\pi_1)^{(4k+2)}=Q(\pi_2)^{(4k+2)}=Q(\pi')$, where $\pi'$ is of the form $(2^2,1^6)$, 
and $Q(\pi_1)^{(4k)}=Q(\pi_2)^{(4k)}$ is a trivial quandle. Hence, we cannot distinguish them by (iii). 

Now, we apply the strategy (iv) for those quandles. Then we can compute 
\begin{align*}
&|C_{\Sf_{10}}(\pi_1) \cap \Af_{10} \setminus \Sf_{10} /C_{\Sf_{10}}(\pi_1)|=240; \text{ and }\\
&|C_{\Sf_{10}}(\pi_2) \cap \Af_{10} \setminus \Sf_{10} /C_{\Sf_{10}}(\pi_2)|=291. 
\end{align*}
Note that we calculate these numbers by using {\tt GAP}. See \cite{GAP}. 

Therefore, we conclude that $Q(\pi_1) \not\cong_Q Q(\pi_2)$. 

\bigskip

\noindent
\underline{$n=13,14,17,18,21,22,24,25,26,27,29,30$}

In these cases, similar to the case $n=10$, we can distinguish $Q(\pi)$'s by the strategies (i), (ii), (iii) and (iv).


\bigskip

Consequently, we obtain the following. 
\begin{thm}\label{thm:sym}
We have a one-to-one correspondence between $\Qc(\Sf_n)$ and the conjugacy classes of $\auto(\Sf_n)$ for any $n \in \{3,4,\ldots,30\} \setminus \{15\}$. 
In particular, we have a one-to-one correspondence between $\Qc(\Sf_n)$ and the conjugacy classes of $\Sf_n$ for $n \in \{3,4,\ldots,30\}\setminus \{6,15\}$. 
\end{thm}

\noindent
\underline{$n=15$}

We encounter the problem in the case $n=15$. We can distinguish $Q(\pi)$'s except for $(9,3^2)$ and $(9,3,1^3)$. 
Let $\pi_1$ be of the shape $(9,3^2)$ and let $\pi_2$ be of the shape $(9,3,1^3)$. 
Then the quandles $Q(\pi_1)$ and $Q(\pi_2)$ have the invariants as follows: 
$$\ord(\pi_1)=\ord(\pi_2)=9, \;\; |C_{\Sf_{15}}(\pi_1)|=|C_{\Sf_{15}}(\pi_2)|=162, \;\; 
\text{(parity of $\pi_1$)}=\text{(parity of $\pi_2$)}=\text{even}. $$
We can also see that we cannot distinguish them by the strategy (iii). Moreover, we have 
$$|C_{\Sf_{15}}(\pi_1) \cap \Af_{15} \setminus \Sf_{15} /C_{\Sf_{15}}(\pi_1)|=
|C_{\Sf_{15}}(\pi_2) \cap \Af_{15} \setminus \Sf_{15} /C_{\Sf_{15}}(\pi_2)|=101,415,520. $$
Therefore, we cannot distinguish them by (iv). 

On the other hand, we can calculate that 
$$|C_{\Sf_{15}}(\pi_1) \setminus \Sf_{15} /C_{\Sf_{15}}(\pi_1)|=50,716,744 \text{ and }
|C_{\Sf_{15}}(\pi_2) \setminus \Sf_{15} /C_{\Sf_{15}}(\pi_2)|=55,008,600, $$
although we do not know if $|C_{\Sf_n}(\pi) \setminus \Sf_n / C_{\Sf_n}(\pi)|$ is an invariant of $Q(\pi)$. 


\smallskip

Finally, we conclude the present paper by suggesting the following question. 
\begin{q}\label{yosou}
For any $n \neq 2,6$, does $\Qc(\Sf_n)$ one-to-one correspond to the set of the conjugacy classes of of $\Sf_n$? 
\end{q}
Theorem~\ref{thm:sym} says that this is true when $n \in \{3,4,\ldots,30\}\setminus \{6,15\}$. 
Note that this is not true in the case $\Qc(C_n)$. See Remark~\ref{rem:cyclic}. 


\bigskip


\begin{thebibliography}{10}
    \providecommand{\url}[1]{\texttt{#1}}
    \providecommand{\urlprefix}{\relax}
    \providecommand{\selectlanguage}[1]{\relax}
    \providecommand{\eprint}[2][arXiv]{\textt{#1:#2}}
    \providecommand*{\ZBL}[2][Zbl]{#1~#2}
    \providecommand*{\Zbl}[1]{\ZBL[Zbl]{#1}}
    \providecommand*{\JFM}[1]{\ZBL[JFM]{#1}}
    \providecommand*{\ERAM}[1]{\ZBL[ERAM]{#1}}
    
    \bibitem{GAP}
    GAP, {\tt https://www.gap-system.org/}
    
    \bibitem{Holder}
    O.~H\"{o}lder, Bildung zusammengesetzter {G}ruppen. \emph{Math. Ann.}
      \textbf{46} (1895), no.~3, \mbox{321--422}.    \MR{1510888}
      \urlprefix\url{https://doi.org/10.1007/BF01450217}
    
    \bibitem{IT}
    Y.~Ishihara and H.~Tamaru, Flat connected finite quandles. \emph{Proc. Amer.
      Math. Soc.} \textbf{144} (2016), no.~11, \mbox{4959--4971}.    \MR{3544543}
      \urlprefix\url{https://doi.org/10.1090/proc/13095}
    
    \bibitem{Joyce}
    D.~Joyce, A classifying invariant of knots, the knot quandle. \emph{J. Pure
      Appl. Algebra} \textbf{23} (1982), no.~1, \mbox{37--65}.    \MR{638121}
      \urlprefix\url{https://doi.org/10.1016/0022-4049(82)90077-9}
    
    \bibitem{KTW}
    S.~Kamada, H.~Tamaru, and K.~Wada, On classification of quandles of cyclic
      type. \emph{Tokyo J. Math.} \textbf{39} (2016), no.~1, \mbox{157--171}.
      \MR{3543136} \urlprefix\url{https://doi.org/10.3836/tjm/1459367262}
    
    \bibitem{Lam1993}
    T.~Y. Lam and D.~B. Leep, Combinatorial structure on the automorphism group of
      {$S_6$}. \emph{Exposition. Math.} \textbf{11} (1993), no.~4, \mbox{289--308}.
         \MR{1240362}
    
    \bibitem{Nelson}
    S.~Nelson, Classification of finite {A}lexander quandles. In \emph{Proceedings
      of the {S}pring {T}opology and {D}ynamical {S}ystems {C}onference}. 2003,
      \mbox{245--258}    \MR{2048935}
    
    \bibitem{Roman}
    S.~Roman, \emph{Fundamentals of group theory}. An advanced approach.
      Birkh\"{a}user/Springer, New York, 2012.    \MR{2866265}
      \urlprefix\url{https://doi.org/10.1007/978-0-8176-8301-6}
    
    \bibitem{Rotman}
    J.~J. Rotman, \emph{An introduction to the theory of groups}. Fourth ed.,
      Graduate Texts in Mathematics, 148. Springer-Verlag, New York, 1995.
      \MR{1307623} \urlprefix\url{https://doi.org/10.1007/978-1-4612-4176-8}
    
    \bibitem{Scott}
    W.~R. Scott, \emph{Group theory}. Second ed. Dover Publications, Inc., New
      York, 1987.    \MR{896269}
    
    \bibitem{Ven}
    L.~Vendramin, Doubly transitive groups and cyclic quandles. \emph{J. Math. Soc.
      Japan} \textbf{69} (2017), no.~3, \mbox{1051--1057}.    \MR{3685034}
      \urlprefix\url{https://doi.org/10.2969/jmsj/06931051}
    
    \bibitem{Wildon2008}
    M.~Wildon, Labelling the character tables of symmetric and alternating groups.
      \emph{Q. J. Math.} \textbf{59} (2008), no.~1, \mbox{123--135}.
      \MR{2392503} \urlprefix\url{https://doi.org/10.1093/qmath/ham024}
    
    \end{thebibliography}

\end{document}